\documentclass[11pt,reqno]{amsart}

\usepackage[T1]{fontenc}
\usepackage{lmodern}
\usepackage{microtype}
\usepackage{amsmath,amssymb,amsthm,mathtools,mathrsfs}
\usepackage{bm}
\usepackage{enumitem}
\usepackage{booktabs}
\usepackage{array}
\usepackage{xcolor}
\usepackage[colorlinks=true,linkcolor=blue!55!black,citecolor=blue!55!black,urlcolor=blue!55!black]{hyperref}
\usepackage{aliascnt}
\usepackage[nameinlink,capitalize,noabbrev]{cleveref}
\usepackage[margin=1.08in]{geometry}

\allowdisplaybreaks
\numberwithin{equation}{section}

\theoremstyle{plain}
\newtheorem{theorem}{Theorem}[section]
\newaliascnt{proposition}{theorem}
\newtheorem{proposition}[proposition]{Proposition}
\aliascntresetthe{proposition}
\newaliascnt{lemma}{theorem}
\newtheorem{lemma}[lemma]{Lemma}
\aliascntresetthe{lemma}
\newaliascnt{corollary}{theorem}
\newtheorem{corollary}[corollary]{Corollary}
\aliascntresetthe{corollary}
\newaliascnt{claim}{theorem}

\aliascntresetthe{claim}
\theoremstyle{definition}
\newaliascnt{definition}{theorem}
\newtheorem{definition}[definition]{Definition}
\aliascntresetthe{definition}
\theoremstyle{remark}
\newaliascnt{remark}{theorem}
\newtheorem{remark}[remark]{Remark}
\aliascntresetthe{remark}

\crefname{theorem}{Theorem}{Theorems}
\crefname{proposition}{Proposition}{Propositions}
\crefname{lemma}{Lemma}{Lemmas}
\crefname{corollary}{Corollary}{Corollaries}
\crefname{definition}{Definition}{Definitions}
\crefname{remark}{Remark}{Remarks}
\crefname{claim}{Claim}{Claims}

\newcommand{\T}{\mathbb T}
\newcommand{\R}{\mathbb R}
\newcommand{\Z}{\mathbb Z}
\newcommand{\N}{\mathbb N}
\newcommand{\Sym}{\operatorname{Sym}}
\newcommand{\Szero}{\mathcal S^{3\times 3}_0}

\newcommand{\diver}{\operatorname{div}}
\newcommand{\supp}{\operatorname{supp}}
\newcommand{\dev}{\operatorname{dev}}

\newcommand{\Rop}{\mathcal R_{\T}}
\newcommand{\Sloc}{\mathscr S_U}
\newcommand{\dd}{\,\mathrm d}
\newcommand{\eps}{\varepsilon}
\newcommand{\calD}{\mathcal D}
\newcommand{\calE}{\mathcal E}

\newcommand{\calO}{\mathcal O}
\newcommand{\one}{\mathbf 1}
\newcommand{\esssupp}{\operatorname*{ess\,supp}}
\newcommand{\AmpNorm}[2]{\mathfrak A_{#1}(#2)}

\title[Compactly supported stationary Navier--Stokes flows]{Compactly Supported Finite-Energy Stationary Solutions of the Three-Dimensional Navier--Stokes Equations}
\author{Xuanxuan ZHAO}
\address{School of Mathematical Sciences, Fudan University,220 Handan Road, Shanghai 200433, China}
\email{xxzhao@math.cuhk.edu.hk}
\subjclass[2020]{35Q30, 35A02, 35D30, 76D05}
\keywords{stationary Navier--Stokes equations, weak solutions, convex integration, compact support, logarithmic Mikado flows, localized h-principle, Bogovskii operator}
\hypersetup{
  pdftitle={Compactly Supported Finite-Energy Stationary Solutions of the Three-Dimensional Navier--Stokes Equations},
  pdfauthor={Xuanxuan Zhao},
  pdfsubject={Compactly supported stationary weak solutions in three dimensions},
  pdfkeywords={stationary Navier--Stokes equations, weak solutions, compact support, convex integration, logarithmic Mikado flows, localized h-principle, Leray-Hopf solutions}
}

\begin{document}

\begin{abstract}
We construct nonzero compactly supported stationary distributional solutions
\[
 u\in L^2(\R^3;\R^3)
\]
of the unforced incompressible Navier--Stokes equations, together with compactly supported pressures in $L^1$, with both norms arbitrarily small.  The construction addresses the three-dimensional finite-energy endpoint at which the usual single-scale intermittent Mikado mechanism loses its algebraic gain against the Laplacian.  The main new ingredient is a logarithmic Mikado profile, built from a truncated two-dimensional harmonic dipole and spread over logarithmically many transverse scales.  The same perturbation yields a localized $h$-principle: nonzero stationary solutions are strongly dense in the localized $L^p$ classes for every $1\le p<2$ and in $H^{-1}$, and weakly dense at $p=2$, while strong $L^2$ density fails because of an exact isotropic quadratic-moment constraint.  We also prescribe arbitrary positive $L^2$ norms, periodize the construction to $\T^3$, and obtain a stationary-versus-Leray nonuniqueness mechanism for the evolutionary equations.
\end{abstract}

\maketitle

\section{Introduction}\label{sec:intro}

We consider the stationary incompressible Navier--Stokes equations with viscosity one on the whole space,
\begin{equation}\label{eq:NS}
 -\Delta u+\diver(u\otimes u)+\nabla p=0,
 \qquad \diver u=0
 \quad\text{in }\R^3.
\end{equation}
For $u\in L^2(\R^3)$ the tensor $u\otimes u$ belongs to $L^1$, and hence the standard distributional formulation is meaningful.

\begin{definition}\label{def:weak}
A vector field $u\in L^2(\R^3;\R^3)$ is a stationary weak solution of \eqref{eq:NS} if $\diver u=0$ in distributions and
\begin{equation}\label{eq:weak-form}
 \int_{\R^3}u\cdot(-\Delta\varphi)\,\dd x
 -\int_{\R^3}(u\otimes u):\nabla\varphi\,\dd x=0
\end{equation}
for every $\varphi\in C_c^\infty(\R^3;\R^3)$ satisfying $\diver\varphi=0$.
\end{definition}

Throughout the statements below, $I$ denotes the $3\times3$ identity matrix and
\[
 \dev A:=A-\frac13(\operatorname{tr}A)I.
\]
We write $C_{c,\sigma}^\infty(\Omega;\R^3)$ for the smooth compactly supported divergence-free vector fields on $\Omega$.  For nonsmooth fields, support means essential support.  The unit torus $\T^3=\R^3/\Z^3$ is equipped with normalized Haar measure, and $C_{\mathrm w}([0,\infty);L^2)$ denotes continuity with respect to the weak $L^2$ topology.

The main result of the manuscript is the following.

\begin{theorem}\label{thm:main}
Let $B\subset\R^3$ be a nonempty open ball.  For every $\eps>0$ there exists a nonzero pair
\[
 u\in L^2(\R^3;\R^3),
 \qquad p\in L^1(\R^3),
\]
such that
\[
 \esssupp u\cup\esssupp p\Subset B,
 \qquad
 \|u\|_{L^2(\R^3)}+\|p\|_{L^1(\R^3)}<\eps,
\]
and
\[
 -\Delta u+\diver(u\otimes u)+\nabla p=0,
 \qquad \diver u=0
\]
in $\calD'(\R^3)$.  In particular, $u$ is a stationary weak solution in the sense of \cref{def:weak}.
\end{theorem}

\subsection{Further main results}\label{subsec:further-results}

The perturbation estimate contains considerably more flexibility than is needed for the existence statement alone.  We state the principal consequences here and defer their proofs to \cref{sec:flexibility,sec:affine-energy,sec:periodic-dynamic}.

For a ball $B\subset\R^3$ and $1\le r\le2$, set
\[
 \mathcal X_r(B)
 :=\overline{C_{c,\sigma}^\infty(B;\R^3)}^{\,L^r(\R^3)},
 \qquad
 \mathcal X_{-1}(B)
 :=\overline{C_{c,\sigma}^\infty(B;\R^3)}^{\,H^{-1}(\R^3)}.
\]
Let $\mathscr S(B)$ denote the set of velocities $u\in L^2(\R^3;\R^3)$ for which there exists a pressure $p\in L^1(\R^3)$ such that
\[
 \esssupp u\cup\esssupp p\Subset B
\]
and $(u,p)$ solves \eqref{eq:NS} in distributions.  Every $u\in\mathscr S(B)$ belongs to $\mathcal X_2(B)$: since its support is compactly contained in $B$, standard mollification preserves divergence and, for sufficiently small mollification scale, keeps the support inside $B$.

\begin{theorem}[Localized flexibility and endpoint rigidity]\label{thm:flexibility-rigidity}
Let $B\subset\R^3$ be a nonempty open ball.
\begin{enumerate}[label=\textup{(\roman*)}]
 \item For every $1\le r<2$, the nonzero elements of $\mathscr S(B)$ are strongly dense in $\mathcal X_r(B)$.
 \item The nonzero elements of $\mathscr S(B)$ are strongly dense in $\mathcal X_{-1}(B)$ and weakly dense in $\mathcal X_2(B)$.
 \item Every $u\in\mathscr S(B)$ satisfies
 \begin{equation}\label{eq:isotropic-moment}
  \int_{\R^3}u\otimes u\,\dd x
  =\frac13\|u\|_{L^2(\R^3)}^2I.
 \end{equation}
 Consequently, the strong $L^2$ closure of $\mathscr S(B)$ is contained in
 \[
  \left\{v\in\mathcal X_2(B):
  \dev\int_{\R^3}v\otimes v\,\dd x=0\right\},
 \]
 and $\mathscr S(B)$ is not strongly dense in $\mathcal X_2(B)$.
\end{enumerate}
\end{theorem}

\begin{proposition}[Full affine stress identity]\label{thm:full-affine-identity}
Let $u\in L^2(\R^3;\R^3)$ and $p\in L^1(\R^3)$ be compactly supported and solve \eqref{eq:NS} in distributions.  Then
\begin{equation}\label{eq:full-affine-identity}
 \int_{\R^3}\bigl(u\otimes u+pI\bigr)\,\dd x=0.
\end{equation}
In particular,
\begin{equation}\label{eq:pressure-mass}
 \int_{\R^3}u\otimes u\,\dd x
 =\frac13\|u\|_2^2I,
 \qquad
 \int_{\R^3}p\,\dd x=-\frac13\|u\|_2^2,
\end{equation}
and hence
\begin{equation}\label{eq:pressure-lower-bound}
 \|p\|_{L^1(\R^3)}\ge\frac13\|u\|_{L^2(\R^3)}^2.
\end{equation}
\end{proposition}

\begin{corollary}[Prescribed energy]\label{thm:exact-energy}
Let $\Omega\subset\R^3$ be a nonempty open set and let $E>0$.  There are uncountably many stationary pairs $(u,p)$ with
\[
 \esssupp u\cup\esssupp p\Subset\Omega,
 \qquad p\in L^1(\R^3),
 \qquad \|u\|_{L^2(\R^3)}=E.
\]
For a family produced in the proof, there is a constant $C_\Omega>0$, independent of $E$, such that
\begin{equation}\label{eq:pressure-energy-bounds}
 \frac13E^2\le\|p\|_{L^1(\R^3)}\le C_\Omega E^2.
\end{equation}
\end{corollary}

\begin{corollary}[Periodic endpoint solutions]\label{cor:periodic-endpoint}
For every $E>0$ there exists a nonconstant zero-mean stationary solution
\[
 u\in L^2(\T^3;\R^3),
 \qquad p\in L^1(\T^3),
 \qquad \|u\|_{L^2(\T^3)}=E.
\]
In particular, nonzero periodic stationary solutions exist with arbitrarily small $L^2$ norm.  Moreover, for every $1\le r<2$, a periodic solenoidal target supported in a ball compactly contained in a fundamental cube can be approximated strongly in $L^r(\T^3)$ by nonzero periodic stationary solutions.
\end{corollary}

\begin{corollary}[Stationary-versus-Leray nonuniqueness]\label{thm:stationary-leray}
There exist arbitrarily small, as well as exactly prescribed-energy, initial data in $L^2(\R^3)$ for which the unforced evolutionary Navier--Stokes equations admit at least two distinct global distributional weak solutions in
\[
 L^\infty\bigl([0,\infty);L^2(\R^3;\R^3)\bigr)
 \cap C_{\mathrm w}\bigl([0,\infty);L^2(\R^3;\R^3)\bigr).
\]
One branch is stationary and the other is a Leray--Hopf solution; the two branches are distinct at every positive time, and the Leray--Hopf branch $U^{\rm L}$ satisfies
\begin{equation}\label{eq:strict-energy-drop}
 \|U^{\rm L}(t)\|_{L^2}<\|u_{\rm in}\|_{L^2}
 \qquad\text{for every }t>0.
\end{equation}
The same conclusion holds on $\T^3$.  For every ball $B\subset\R^3$ and every $1\le r<2$, initial data generating this mechanism are strongly dense in $\mathcal X_r(B)$; at the endpoint $r=2$ they are weakly dense in $\mathcal X_2(B)$.
\end{corollary}

Any solution in \cref{thm:main} must lie outside the energy class.  Indeed, if $u\in H^1(\R^3)$ were a stationary weak solution, then $u\in L^4(\R^3)$ by interpolation and $C_{c,\sigma}^\infty(\R^3)$ is dense in the solenoidal subspace of $H^1$.  Approximating $u$ by admissible test fields in \eqref{eq:weak-form} gives
\[
 \int_{\R^3}|\nabla u|^2\,\dd x
 -\int_{\R^3}(u\otimes u):\nabla u\,\dd x=0.
\]
The second integral vanishes by incompressibility, and hence $u=0$.  No such self-testing argument is available for a bare $L^2$ distributional solution.

The proof of \cref{thm:main} is based on the method of convex integration.  The origins of this method lie in differential geometry, beginning with Nash's construction of flexible $C^1$ isometric embeddings and Gromov's subsequent theory of convex integration for underdetermined differential relations and the $h$-principle \cite{Nash1954,Gromov1973,Gromov1986}.  Its characteristic feature is the use of rapidly oscillating perturbations to eliminate a nonlinear defect while retaining only weak control of the approximate solutions.  This geometric flexibility mechanism has become a central tool in the study of nonuniqueness, anomalous dissipation, and critical regularity for nonlinear PDEs.

In incompressible fluid dynamics, highly irregular weak Euler flows had already appeared in the works of Scheffer and Shnirelman \cite{Scheffer1993,Shnirelman1997}.  A systematic convex-integration approach was initiated by De Lellis and Sz\'ekelyhidi, who reformulated the Euler equations as a differential inclusion and developed a quantitative flexibility theory for weak solutions \cite{DLS2009}.  Subsequent works connected this theory with Onsager's conjecture and progressively improved the regularity of dissipative Euler flows \cite{DLS2013,DLS2014,BDIS2015}.  Daneri and Sz\'ekelyhidi introduced the stationary perturbation profiles now known as Mikado flows \cite{DaneriSzekelyhidi2017}, which became an important ingredient in later high-regularity schemes.  Isett ultimately proved the flexible part of Onsager's conjecture by constructing dissipative Euler solutions in $C^\alpha$ for every $\alpha<1/3$ \cite{Isett2018}.

The method was subsequently adapted to equations with dissipation.  Buckmaster and Vicol introduced highly intermittent building blocks to construct nonunique finite-energy weak solutions of the three-dimensional Navier--Stokes equations \cite{BuckmasterVicol2019}.  In the viscous setting, oscillation must be combined with concentration in order to make the Laplacian perturbative.  This principle led to nonuniqueness for hyperviscous equations below the Lions exponent and to sharp results near the Ladyzhenskaya--Prodi--Serrin threshold \cite{LuoTiti2020,CheskidovLuo2022}.  Whole-space convex-integration constructions have subsequently established both finite-energy and sharp nonuniqueness results on $\R^3$ \cite{MiaoNieYeFiniteEnergy2024,MiaoNieYe2024}.  Nonuniqueness of Leray solutions has been established by a different instability mechanism for the forced Navier--Stokes problem \cite{AlbrittonBrueColombo2022}.

Convex integration has continued to develop rapidly in recent years, both through new perturbative architectures and through applications to a growing range of nonlinear PDEs.  Recent developments include flexibility for two-dimensional Euler flows with integrable vorticity \cite{BrueColomboKumar2024}, the continuity equation beyond the Sobolev embedding threshold \cite{ColomboColomboKumar2025}, global dissipative solutions for the three-dimensional Navier--Stokes and MHD equations \cite{CheskidovZengZhang2025}, Onsager-type results and dissipative constructions for ideal MHD \cite{MiaoNieYe2025,EncisoPenafielPeraltaMHD2025,GiardiSzekelyhidi2026}, topology-preserving convex integration for steady three-dimensional Euler flows and extension results for weak Euler solutions \cite{EncisoPenafielPeraltaSteady2025,EncisoPenafielPeraltaExtension2024}, the Newton--Nash approach to the two-dimensional Onsager conjecture \cite{GiriRadu2024}, convex-integration solutions of periodic gKdV \cite{GismondiMaPathakRadu2026}, Onsager-type and nonuniqueness results for SQG and more general active scalar equations \cite{ZhaoActiveScalar2024,DaiGiriRadu2024,LooiIsett2024,ChengKwonLi2021}, non-conservation of generalized helicity for Euler flows \cite{GiriKwonNovack2026}, nonuniqueness phenomena in nonlinear elastodynamics \cite{MaoQuLame2025,MaoQuElastodynamics2025}, and stationary or hypo-dissipative MHD constructions \cite{QuZhang2026,LiQu2026}.  These works illustrate the continuing expansion of convex integration beyond its original Euler setting and the development of new mechanisms for overcoming critical regularity and scaling barriers.

The stationary Navier--Stokes problem presents a different form of the flexibility question, since temporal intermittency, gluing in time, and other dynamical mechanisms are unavailable.  Luo established the first stationary finite-energy convex-integration construction for the unforced Navier--Stokes equations on $\T^d$ in dimensions $d\ge4$ \cite{Luo2019}.  In fact, his construction yields
\[
 u\in H^\beta(\T^d)
 \qquad\text{for every }\beta<\frac1{200},
\]
and therefore gives slightly more regularity than finite energy alone.  Recent work of Cheskidov and Hou constructs stationary singular solutions in lower-regularity regimes \cite{CheskidovHou2026}, while Fujii's whole-space work is formulated in scaling-critical Besov classes \cite{Fujii2026}.

Very recently, Fujii constructed non-unique stationary solutions in $L^2(\mathbb R^2)$ by exploiting the additional freedom of choosing a suitable arbitrarily small external force \cite{FujiiL2}. More precisely, his construction first produces a nontrivial Oseen mode around a small background flow and then uses the remaining equation to determine the common forcing. In comparison, the method developed in the present paper is sufficiently robust to close the construction in the genuinely unforced setting $f=0$, without relying on this additional degree of freedom.

These results demonstrate substantial stationary flexibility but do not provide the compactly supported ordinary $L^2(\R^3)$ solution considered here.

\section{Overview of the proof and the endpoint mechanism}

We now explain the mechanism behind Theorem \ref{thm:main} and, in particular, why the three-dimensional $L^2$ case is an endpoint for the usual intermittent Mikado construction.  The key point is that, at this scaling, ordinary single-scale concentration no longer produces a small viscous stress.  Our construction replaces this missing algebraic gain by a logarithmic one, obtained from a multiscale transverse profile.  We first describe the endpoint obstruction for standard Mikado tubes and then explain how the logarithmic profile overcomes it.

The endpoint obstruction can already be seen from the elementary scaling of an ordinary intermittent Mikado tube.  Consider an $L^2$-normalized tube in dimension $d$, concentrated in the $d-1$ transverse directions at scale $\mu^{-1}$.  Its typical amplitude and support volume satisfy
\[
 |W_\mu|\sim\mu^{(d-1)/2},
 \qquad
 |\supp W_\mu|\sim\mu^{-(d-1)},
\]
so that
\[
 \|W_\mu\|_{L^2}\sim1,
 \qquad
 \|W_\mu\|_{L^1}\sim\mu^{-(d-1)/2}.
\]
The viscous term is more expensive.  For a standard single-scale Mikado profile, the natural stress representing $\Delta W_\mu$ has $L^1$ size on the scale of one spatial derivative of the profile, namely
\[
 \|\nabla W_\mu\|_{L^1}
 \sim
 \mu^{1-\frac{d-1}{2}}
 =\mu^{\frac{3-d}{2}}.
\]
Thus the usual single-scale intermittency mechanism produces an algebraic viscous gain in dimensions $d\ge4$, whereas this gain disappears exactly in dimension three.

Indeed, when $d=3$,
\[
 \|W_\mu\|_{L^2}\sim1,
 \qquad
 \|W_\mu\|_{L^1}\sim\mu^{-1},
 \qquad
 \|\nabla W_\mu\|_{L^1}\sim1.
\]
Increasing the concentration parameter therefore does not make the stress associated with the Laplacian smaller.  The three-dimensional finite-energy stationary problem is consequently an endpoint problem: the ordinary power-law gain supplied by spatial concentration has been exhausted.  Rather than seeking a stronger negative power of $\mu$, one needs a different source of smallness which survives at this scale-invariant endpoint.

The main new ingredient is a logarithmically multiscale transverse profile. In two dimensions let
\[
G(y)=\frac{1}{2\pi}\log |y|
\]
be the fundamental solution of the Laplacian and consider a mollified dipole
\[
e\cdot\nabla(G*\rho_\varepsilon).
\]
Away from its mollified core this profile is exactly
\[
\frac{e\cdot y}{2\pi |y|^2},
\]
and therefore has size $r^{-1}$. On the annulus
\[
\varepsilon\lesssim r\lesssim 1
\]
its squared $L^2$ mass is proportional to
\[
\int_\varepsilon^1 \frac{dr}{r}
=
\log\frac{1}{\varepsilon}.
\]
At the same time the profile is harmonic throughout this long annulus: its Laplacian is confined to the mollified core and to a fixed outer cutoff region. Taking
\[
\varepsilon=e^{-L}
\]
and normalizing by the factor $\sqrt{L}$, we obtain a transverse field $\psi_L$ satisfying
\[
\|\psi_L\|_{L^2(\mathbb R^2)}=1,
\qquad
\Delta\psi_L=\operatorname{div}F_L,
\]
while
\[
\|\psi_L\|_{L^1(\mathbb R^2)}
+
\|F_L\|_{L^1(\mathbb R^2)}
\lesssim L^{-1/2}.
\]
The construction also provides a lower-order potential $Q_L$ with
\[
\psi_L=\operatorname{div}Q_L
\]
and the same logarithmic smallness in the norms required later. These are precisely the estimates established in Lemma \ref{lem:dipole} of the present paper.

This logarithmic gain may be viewed as the endpoint replacement for ordinary intermittency. Formally,
\[
\psi_L=\Delta^{-1}\operatorname{div}F_L,
\]
and in two dimensions the operator on the right has the scaling of the Riesz potential $I_1$. The strong endpoint mapping
\[
I_1:L^1(\mathbb R^2)\longrightarrow L^2(\mathbb R^2)
\]
fails: only a weak-$L^2$ endpoint estimate is available. The dipole above exploits exactly this failure. Before normalization, its $L^2$ norm grows like
\[
\sqrt{\log(1/\varepsilon)},
\]
although the $L^1$ size of the associated flux remains bounded. After normalization, this becomes
\[
\|\psi_L\|_2\sim 1,
\qquad
\|F_L\|_1\lesssim L^{-1/2}.
\]
In contrast to an ordinary intermittent Mikado flow, whose energy is concentrated near one transverse scale, the present profile distributes comparable amounts of $L^2$ mass over logarithmically many scales. The gain is therefore logarithmic rather than algebraic.

Lifting these transverse profiles to disjoint rational tubes on $\mathbb T^3$ produces logarithmic Mikado fields $W_{k,L}$ with the usual Euler geometry,
\[
\operatorname{div}W_{k,L}=0,
\qquad
\operatorname{div}(W_{k,L}\otimes W_{k,L})=0,
\qquad
\int_{\mathbb T^3}W_{k,L}\otimes W_{k,L}\,dx
=
k\otimes k,
\]
but with an additional viscous structure
\[
\operatorname{div}S_{k,L}=\Delta W_{k,L}
\]
such that
\[
\|W_{k,L}\|_{L^2(\mathbb T^3)}=|k|,
\qquad
\|W_{k,L}\|_{L^1(\mathbb T^3)}
+
\|S_{k,L}\|_{L^1(\mathbb T^3)}
\lesssim L^{-1/2}.
\]
Thus the $L^2$ cost of the perturbation remains of order one while the stress needed to absorb the viscosity tends to zero.  This separation is the basic reason that the stationary $L^2$ iteration can close.

Two additional ingredients allow the logarithmic gain to survive in the whole-space iteration.  First, an antisymmetric potential provides an exactly divergence-free compactly supported localization of the periodic microscopic profiles, avoiding the nonlocal whole-space Leray projection.  Second, because the innermost scale is $e^{-L}$, positive derivatives of the new profile can be extremely large; a direct modulation estimate would destroy the logarithmic gain.  The modulated-viscosity identity of \cref{sec:viscosity} cancels all fast derivatives of the Mikado profile and leaves estimates involving only its low-order $L^1$ quantities.  The remaining local errors are handled by a support-preserving symmetric anti-divergence operator.

Combining these ingredients yields a compactly supported Navier--Stokes--Reynolds perturbation whose new stress satisfies
\[
 \|\overline R\|_{L^1(\R^3)}
 \le C\left(\gamma^{-1}+\frac{\gamma}{\sqrt L}\right),
\]
where the constant is independent of $L$, $\gamma$, and the internal scale $e^{-L}$.  One may therefore first choose $\gamma$ large and then take $L\gg\gamma^2$.  At the same time the velocity increments are summable in $L^2$, so the iteration converges strongly in $L^2$ and the quadratic nonlinearities converge strongly in $L^1$.  The limiting field is therefore an ordinary distributional solution with $u\otimes u\in L^1$, rather than a stationary singular solution requiring a renormalized product.

The paper is organized as follows.  \Cref{sec:prelim} collects the geometric and operator-theoretic tools, including the explicit finite-rank correction in the local symmetric anti-divergence operator.  The logarithmic dipole and periodic Mikado family are constructed in \cref{sec:dipole,sec:mikado}.  Compact localization and the modulated viscosity identity are proved in \cref{sec:localization,sec:viscosity}.  The whole-space oscillation stress is treated in \cref{sec:oscillation}.  The  perturbation proposition  are given in \cref{sec:perturbation}, and the principal iteration is carried out in \cref{sec:iteration}.  The localized $h$-principle and its sharp $L^2$ obstruction are proved in \cref{sec:flexibility}; affine identities, exact energy prescription, and multiplicity are developed in \cref{sec:affine-energy}; and periodization together with evolutionary consequences is treated in \cref{sec:periodic-dynamic}.  \Cref{app:local-kernel} gives a self-contained proof of the flat local symmetric-divergence construction used in the argument.

\section{Preliminaries}\label{sec:prelim}

\subsection{Notation and the Reynolds system}
We use the Euclidean/Frobenius norm
\[
 |A|=\left(\sum_{i,j=1}^3 A_{ij}^2\right)^{1/2}
\]
for matrices and the associated balls in finite-dimensional matrix spaces.  Set
\[
 \Sym_3=\{A\in\R^{3\times3}:A^T=A\},
 \qquad
 \Szero=\{A\in\Sym_3:\operatorname{tr}A=0\}.
\]
For a vector field $V$ and a matrix field $A=(A_{ij})$ our conventions are
\[
 (\nabla V)_{ij}=\partial_jV_i,
 \qquad
 (\diver A)_i=\partial_jA_{ij}.
\]
Repeated indices are summed.  We recall that
\[
 \dev A=A-\frac13(\operatorname{tr}A)I.
\]
With the Frobenius norm, $\dev$ is an orthogonal projection and hence
\begin{equation}\label{eq:dev-contraction}
 |\dev A|\le |A|,
 \qquad
 \|\dev A\|_{L^1}\le\|A\|_{L^1}.
\end{equation}
Smooth fields retain their usual closed support.  Our Fourier convention on the normalized unit torus is
\[
 \widehat f(m)=\int_{\T^3}f(x)e^{-2\pi i m\cdot x}\,\dd x,
 \qquad
 f(x)=\sum_{m\in\Z^3}\widehat f(m)e^{2\pi i m\cdot x}.
\]
Periodic functions are identified with their $\Z^3$-periodic lifts to $\R^3$ whenever they are evaluated at $\gamma x$.  We always display the domain when periodic and whole-space quantities occur in the same argument: periodic profile norms are written as $\|\cdot\|_{L^p(\T^3)}$, whole-space norms as $\|\cdot\|_{L^p(\R^3)}$, and transverse norms as $\|\cdot\|_{L^p(\R^2)}$.  Within a paragraph devoted entirely to one of these domains, a shorter subscript such as $\|\cdot\|_p$ may be used only after the domain has been fixed explicitly.  For a finite amplitude family $a=(a_k)_{k\in\Lambda}$ we abbreviate
\[
 \AmpNorm{m}{a}:=\max_{k\in\Lambda}\|a_k\|_{C^m(\R^3)}.
\]

We iterate smooth compactly supported solutions of the stationary Navier--Stokes--Reynolds system
\begin{equation}\label{eq:NSR}
 -\Delta u+\diver(u\otimes u)+\nabla p=\diver R,
 \qquad \diver u=0,
\end{equation}
where $R\in C_c^\infty(U;\Szero)$ and $U\Subset\R^3$ is a fixed ball.

\subsection{The geometric lemma}
We use a standard finite rank-one decomposition; see, for example, \cite{DLS2013,DaneriSzekelyhidi2017}.

\begin{lemma}[Geometric lemma]\label{lem:geometric}
There is a finite labeled set $\Lambda\subset\Z^3\setminus\{0\}$ and, for each $k\in\Lambda$, a smooth positive function
\[
 \Gamma_k:B_{1/2}(I)\cap\Sym_3\longrightarrow(0,\infty),
 \qquad k\in\Lambda,
\]
such that
\begin{equation}\label{eq:geometric}
 M=\sum_{k\in\Lambda}\Gamma_k(M)^2k\otimes k
\end{equation}
for every $M\in B_{1/2}(I)\cap\Sym_3$.
\end{lemma}

\subsection{Periodic symmetric anti-divergence}

\begin{lemma}[Periodic symmetric anti-divergence]\label{lem:periodic-antidiv}
For a smooth periodic vector field $f$, define $\Rop f\in C^\infty(\T^3;\Szero)$ by the Fourier symbol, for $m\ne0$,
\begin{equation}\label{eq:Rtorus-symbol}
 \widehat{\Rop f}_{ij}(m)
 =\frac1{2\pi i}\left[
 \frac{\widehat f_i(m)m_j+\widehat f_j(m)m_i}{|m|^2}
 -\frac{\widehat f(m)\cdot m}{2|m|^2}\delta_{ij}
 -\frac{\widehat f(m)\cdot m}{2|m|^4}m_im_j
 \right],
\end{equation}
and set the zero mode equal to zero.  Then
\begin{equation}\label{eq:Rtorus-div}
 \diver\Rop f=f-\int_{\T^3}f,
\end{equation}
and $\Rop$ is convolution with a matrix-valued kernel $K_{\T}\in L^1(\T^3)$.  In particular,
\begin{equation}\label{eq:Rtorus-L1}
 \|\Rop f\|_{L^1(\T^3)}\le C\|f\|_{L^1(\T^3)}.
\end{equation}
Moreover, away from the origin of the torus,
\begin{equation}\label{eq:Rtorus-kernel-pointwise}
 |K_{\T}(x)|\le C\bigl(1+\operatorname{dist}(x,\Z^3)^{-2}\bigr).
\end{equation}
\end{lemma}

\begin{proof}
Symmetry and trace-freeness follow directly from \eqref{eq:Rtorus-symbol}.  Contracting with $2\pi i m_j$ gives $\widehat f_i(m)$ for every $m\ne0$, which proves \eqref{eq:Rtorus-div}.

We prove the endpoint kernel estimate.  Let $M(\xi)$ denote any scalar component of the matrix multiplier in \eqref{eq:Rtorus-symbol}; it is smooth and homogeneous of degree $-1$ on $\R^3\setminus\{0\}$.  Choose $\chi\in C_c^\infty(\{1/2<|\xi|<2\})$ and a smooth low-frequency cutoff so that the nonzero lattice frequencies are decomposed into a finite low-frequency part and the dyadic pieces
\[
 M_q(m)=\chi(2^{-q}m)M(m),
 \qquad q\ge0.
\]
By homogeneity, $M_q(m)=2^{-q}a(2^{-q}m)$ for a fixed smooth compactly supported function $a$ on the annulus.  Define the Euclidean inverse Fourier transform by
\[
 \check a(x)=\int_{\R^3}a(\xi)e^{2\pi i x\cdot\xi}\,\dd\xi.
\]
With the torus convention fixed above, Poisson summation gives the kernel of the $q$th piece exactly as
\[
 K_q(x)=2^{2q}\sum_{\ell\in\Z^3}\check a\bigl(2^q(x+\ell)\bigr).
\]
Since $\check a$ is Schwartz, for every $N>3$,
\begin{equation}\label{eq:dyadic-kernel-bound}
 |K_q(x)|\le C_N2^{2q}\sum_{\ell\in\Z^3}
 \bigl(1+2^q|x+\ell|\bigr)^{-N}.
\end{equation}
Integrating over a fundamental domain and unfolding the lattice sum yields
\[
 \|K_q\|_{L^1(\T^3)}
 \le C_N2^{2q}\int_{\R^3}(1+2^q|x|)^{-N}\,\dd x
 \le C2^{-q}.
\]
Thus $\sum_{q\ge0}K_q$ converges absolutely in $L^1(\T^3)$; the omitted low-frequency kernel is smooth.  Summing \eqref{eq:dyadic-kernel-bound} first over $2^q\operatorname{dist}(x,\Z^3)\le1$ and then over the complementary range proves \eqref{eq:Rtorus-kernel-pointwise}.  Young's inequality gives \eqref{eq:Rtorus-L1}.  This is an order $-1$ convolution estimate, not a strong $L^1$ estimate for an order-zero Calder\'on--Zygmund operator.
\end{proof}

\subsection{Profile-independent periodic cell estimates}
The following elementary estimates are used repeatedly to avoid constants depending on the internal scale of the profile.

\begin{lemma}[Periodic cells]\label{lem:periodic-cells}
Let $K\subset\R^3$ be bounded, $1\le p<\infty$, $F\in L^p(\T^3)$, and $\gamma\in\N$.  Then
\begin{equation}\label{eq:cell-Lp}
 \|F(\gamma\,\cdot)\|_{L^p(K)}
 \le C_K\|F\|_{L^p(\T^3)}.
\end{equation}
If $a\in C_c^1(\R^3)$ and $F\in L^2(\T^3)$, then
\begin{equation}\label{eq:cell-average}
 \left|
 \int_{\R^3}a(x)^2|F(\gamma x)|^2\,\dd x
 -\|F\|_{L^2(\T^3)}^2\int_{\R^3}a(x)^2\,\dd x
 \right|
 \le C_{\supp a}\gamma^{-1}\|a\|_{C^1}^2\|F\|_2^2.
\end{equation}
Both constants are independent of all derivative norms of $F$.
\end{lemma}

\begin{proof}
Partition $\R^3$ into cubes
\[
 Q_{m,\gamma}=\gamma^{-1}(m+[0,1)^3),\qquad m\in\Z^3.
\]
The integral of $|F(\gamma x)|^p$ over a full cube is exactly $\gamma^{-3}\|F\|_p^p$.  Only $O_K(\gamma^3)$ cubes meet $K$, proving \eqref{eq:cell-Lp}.

For \eqref{eq:cell-average}, on every cube meeting $\supp a$ freeze $a^2$ at one point of the cube.  The freezing error is at most $C\gamma^{-1}\|a\|_{C^1}^2$ times the full-cell integral of $|F(\gamma x)|^2$.  Summing over $O_{\supp a}(\gamma^3)$ cells gives an $O(\gamma^{-1})$ error.  The frozen sum is a Riemann sum for $\int a^2$ with the same error.  No regularity of $F$ enters.
\end{proof}

\subsection{A compactly supported scalar divergence solver}

\begin{lemma}[Compactly supported scalar divergence solver]\label{lem:scalar-bogovskii}
Let $D\subset\R^d$ be a ball.  If $h\in C_c^\infty(D)$ and $\int_Dh=0$, then there is $\mathcal B_Dh\in C_c^\infty(D;\R^d)$ such that
\[
 \diver(\mathcal B_Dh)=h.
\]
\end{lemma}

\begin{proof}
This is the classical Bogovskii construction on a star-shaped domain; see, for example, \cite[Chapter~III, Theorem~3.1]{Galdi2011}.  We use this lemma only for two fixed smooth functions, so no endpoint norm estimate is needed here.
\end{proof}

\subsection{Compactly supported symmetric anti-divergence}
Let $U$ be a ball.  A vector field $f\in C_c^\infty(U;\R^3)$ is said to satisfy the force and torque compatibility conditions if
\begin{equation}\label{eq:compatibility}
 \int_{\R^3}f\,\dd x=0,
 \qquad
 \int_{\R^3}(x_if_j-x_jf_i)\,\dd x=0
 \quad(1\le i<j\le3).
\end{equation}
These are exactly the orthogonality conditions against the Euclidean Killing fields, namely translations and rotations.  Fix once and for all $\eta\in C_c^\infty(U)$ with $\int_{\R^3}\eta=1$, set $A_\eta:=\supp\eta\Subset U$, and write $\operatorname{co}$ for the convex hull.

\begin{lemma}[Local symmetric anti-divergence]\label{lem:local-sym-div}
For every ball $U\subset\R^3$ there is a linear operator
\[
 \Sloc:\{f\in C_c^\infty(U;\R^3):f\text{ satisfies \eqref{eq:compatibility}}\}
 \longrightarrow C_c^\infty(U;\Sym_3)
\]
such that
\begin{equation}\label{eq:local-right-inverse}
 \diver\Sloc f=f
\end{equation}
and
\begin{equation}\label{eq:local-L1}
 \|\Sloc f\|_{L^1(\R^3)}\le C_U\|f\|_{L^1(\R^3)}.
\end{equation}
Moreover, the support of $\Sloc f$ is contained in $\operatorname{co}(\supp f\cup A_\eta)\Subset U$.
\end{lemma}

\begin{proof}
Define the straight-line kernel $K_\eta$ and the associated tensor $S_\eta f$ by \eqref{eq:appendix-aeta}--\eqref{eq:appendix-Seta}.  The self-contained calculation in \cref{prop:flat-sym-div} gives
\[
 \diver S_\eta f=f-B_\eta f,
\]
together with symmetry, smoothness, the stated support property, and the estimate
\[
 \|S_\eta f\|_1\le C_U\|f\|_1.
\]
The explicit moment formula \eqref{eq:appendix-Bmoments} reduces, under zero total force, to \eqref{eq:appendix-Btorque}; the torque conditions in \eqref{eq:compatibility} therefore imply $B_\eta f=0$.  We set $\Sloc=S_\eta$ on the compatible subspace.  The output is symmetric but need not be trace-free; its trace is absorbed into the pressure wherever the operator is used.
\end{proof}

\begin{lemma}[Automatic compatibility]\label{lem:auto-compat}
If $A\in C_c^\infty(U;\Sym_3)$ and $f=\diver A$, then $f$ satisfies \eqref{eq:compatibility}.
\end{lemma}

\begin{proof}
Integration by parts gives $\int f_i=0$.  Moreover,
\[
 \int(x_if_j-x_jf_i)
 =-\int A_{ji}+\int A_{ij}=0
\]
because $A$ is symmetric.
\end{proof}

\section{The logarithmic transverse dipole}\label{sec:dipole}

The endpoint gain comes from a two-dimensional dipole which is harmonic on a logarithmically long annulus.

\begin{lemma}[Logarithmic dipole]\label{lem:dipole}
There are $L_0\ge1$ and smooth compactly supported fields
\[
 \psi_L\in C_c^\infty(B_2;\R),
 \qquad
 F_L,Q_L\in C_c^\infty(B_2;\R^2),
 \qquad L\ge L_0,
\]
such that
\begin{align}
 &\int_{\R^2}\psi_L\,\dd y=0,
 \qquad \|\psi_L\|_{L^2(\R^2)}=1,\label{eq:dipole-normal}\\
 &\Delta\psi_L=\diver F_L,
 \qquad \psi_L=\diver Q_L,\label{eq:dipole-potentials}
\end{align}
and
\begin{equation}\label{eq:dipole-estimates}
 \|\psi_L\|_{L^1(\R^2)}+\|F_L\|_{L^1(\R^2)}+\|Q_L\|_{L^2(\R^2)}
 +\|Q_L\|_{L^1(\R^2)}+\|\nabla Q_L\|_{L^1(\R^2)}
 \le CL^{-1/2}.
\end{equation}
The constant is independent of $L$.
\end{lemma}

\begin{proof}
Let
\[
 G(y)=\frac1{2\pi}\log|y|,
 \qquad \Delta G=\delta_0
\]
in $\R^2$.  Choose a nonnegative radial function $\rho\in C_c^\infty(B_{1/8})$ with $\int\rho=1$.  For
\[
 \epsilon=e^{-L}
\]
define
\[
 \rho_\epsilon(y)=\epsilon^{-2}\rho(y/\epsilon),
 \qquad
 \Phi_\epsilon=G*\rho_\epsilon,
 \qquad
 h_\epsilon=e\cdot\nabla\Phi_\epsilon,
 \quad e=(1,0).
\]
Radiality and the two-dimensional Newton theorem give the exact identities
\begin{equation}\label{eq:newton}
 \Phi_\epsilon(y)=G(y),
 \qquad
 h_\epsilon(y)=\frac{e\cdot y}{2\pi|y|^2}
 \quad\text{for }|y|>\epsilon/8.
\end{equation}
For completeness, if $|y|=R>r$, then
\[
 \frac1{2\pi}\int_0^{2\pi}\log|R-re^{i\theta}|\,\dd\theta=\log R;
\]
inserting polar coordinates in $G*\rho_\epsilon$ and using $\int\rho_\epsilon=1$ proves \eqref{eq:newton}.

Let $\chi\in C_c^\infty(B_2)$ be radial and satisfy $\chi=1$ on $B_1$, and set
\[
 \widetilde\psi_\epsilon=\chi h_\epsilon.
\]
It is odd in the first coordinate, hence has zero integral.  On $\epsilon<r<1$,
\[
 |\widetilde\psi_\epsilon(r,\theta)|^2
 =\frac{\cos^2\theta}{4\pi^2r^2}.
\]
Consequently,
\begin{equation}\label{eq:Aeps}
 A_\epsilon^2:=\|\widetilde\psi_\epsilon\|_2^2
 =\frac1{4\pi}\log\frac1\epsilon+O(1)
 =\frac L{4\pi}+O(1).
\end{equation}
The mollified core contributes $O(1)$ by scaling, and the outer cutoff region is fixed.  Similarly,
\begin{equation}\label{eq:psi-tilde-L1}
 \|\widetilde\psi_\epsilon\|_1\le C,
\end{equation}
because the main radial integral is $\int_\epsilon^1\dd r$.

We next construct the viscous flux.  Since $\Delta h_\epsilon=e\cdot\nabla\rho_\epsilon$,
\begin{equation}\label{eq:cutoff-lap}
 \Delta(\chi h_\epsilon)
 =\chi e\cdot\nabla\rho_\epsilon
 +2\nabla\chi\cdot\nabla h_\epsilon+h_\epsilon\Delta\chi.
\end{equation}
The first term equals $\diver(e\rho_\epsilon)$ because $\chi=1$ on $\supp\rho_\epsilon$.  By \eqref{eq:newton}, the last two terms form a fixed function $q_0\in C_c^\infty(B_2\setminus\overline{B_1})$ independent of $\epsilon$ for $L$ large.  Its integral is zero, since both $\Delta\widetilde\psi_\epsilon$ and $\diver(e\rho_\epsilon)$ have zero integral.  By \cref{lem:scalar-bogovskii}, choose the fixed field
\[
 H_0=\mathcal B_{B_2}q_0\in C_c^\infty(B_2;\R^2),
 \qquad \diver H_0=q_0,
\]
and set
\begin{equation}\label{eq:Ftilde}
 \widetilde F_\epsilon=e\rho_\epsilon+H_0.
\end{equation}
Then
\begin{equation}\label{eq:Ftilde-properties}
 \diver\widetilde F_\epsilon=\Delta\widetilde\psi_\epsilon,
 \qquad
 \|\widetilde F_\epsilon\|_1\le C.
\end{equation}

For the lower-order potential,
\[
 \diver(\chi e\Phi_\epsilon)
 =\chi e\cdot\nabla\Phi_\epsilon+(e\cdot\nabla\chi)\Phi_\epsilon
 =\widetilde\psi_\epsilon+b_0.
\]
On the support of $\nabla\chi$, \eqref{eq:newton} shows that $b_0$ is independent of $\epsilon$.  Also $\int b_0=0$, because the integral of the left-hand side and that of $\widetilde\psi_\epsilon$ both vanish.  By \cref{lem:scalar-bogovskii}, choose the fixed field
\[
 Q_0=\mathcal B_{B_2}b_0\in C_c^\infty(B_2;\R^2),
 \qquad \diver Q_0=b_0,
\]
and put
\begin{equation}\label{eq:Qtilde}
 \widetilde Q_\epsilon=\chi e\Phi_\epsilon-Q_0.
\end{equation}
Then
\begin{equation}\label{eq:Qtilde-div}
 \diver\widetilde Q_\epsilon=\widetilde\psi_\epsilon.
\end{equation}

We record uniform estimates needed later.  On the annulus $\epsilon/8<|y|<2$, $\Phi_\epsilon=G$, and hence
\[
 \int_{\epsilon/8<|y|<2}|\Phi_\epsilon|^2+|\Phi_\epsilon|+|\nabla\Phi_\epsilon|\,\dd y\le C.
\]
In the core $|y|\lesssim\epsilon$, scaling gives
\[
 |\Phi_\epsilon(y)|\le C(1+|\log\epsilon|),
 \qquad
 |\nabla\Phi_\epsilon(y)|\le C\epsilon^{-1}.
\]
Thus the core contributions are bounded respectively by
\[
 C\epsilon^2(1+L^2),
 \qquad C\epsilon^2(1+L),
 \qquad C\epsilon.
\]
It follows from \eqref{eq:Qtilde} that
\begin{equation}\label{eq:Qtilde-uniform}
 \|\widetilde Q_\epsilon\|_2
 +\|\widetilde Q_\epsilon\|_1
 +\|\nabla\widetilde Q_\epsilon\|_1
 \le C.
\end{equation}

Finally define
\[
 \psi_L=A_\epsilon^{-1}\widetilde\psi_\epsilon,
 \qquad
 F_L=A_\epsilon^{-1}\widetilde F_\epsilon,
 \qquad
 Q_L=A_\epsilon^{-1}\widetilde Q_\epsilon.
\]
Since $A_\epsilon\simeq\sqrt L$ by \eqref{eq:Aeps}, the identities and estimates in the statement follow from \eqref{eq:psi-tilde-L1}, \eqref{eq:Ftilde-properties}, \eqref{eq:Qtilde-div}, and \eqref{eq:Qtilde-uniform}.
\end{proof}

\begin{remark}\label{rem:logarithmic-mass}
On the main annulus the normalized profile is comparable to $L^{-1/2}r^{-1}\cos\theta$.  Hence every logarithmic radial shell carries comparable $L^2$ mass.  The Laplacian, however, is supported only in the mollified core and the fixed outer cutoff region.  This multiscale separation is the source of the factor $L^{-1/2}$.
\end{remark}

\section{Periodic logarithmic Mikado profiles}\label{sec:mikado}

The building blocks in this section live on the unit torus
\[
 \T^3=\R^3/\Z^3
\]
with normalized Haar measure.  Thus, throughout this section,
\[
 \langle G\rangle_{\T^3}:=\int_{\T^3}G(x)\,\dd x,
\]
and every displayed norm of $W_{k,L}$, $S_{k,L}$, $\Omega_{k,L}$, or $T_{k,L}$ is a norm over $\T^3$.  In later sections these fields are lifted periodically to $\R^3$ and multiplied by compactly supported amplitudes; the resulting modulated fields are then measured over $\R^3$.

\subsection{Disjoint rational geodesics}
For $k\in\Z^3\setminus\{0\}$ define
\[
 H_k:=\{tk\!\!\pmod{\Z^3}:t\in\R\}\subset\T^3.
\]
This is the closed one-dimensional subtorus obtained by projecting the line in direction $k$ to the flat torus.  If $\bar k$ is the primitive integer vector parallel to $k$ and with the same orientation, then $H_k=H_{\bar k}$ as subsets of $\T^3$.

\begin{lemma}\label{lem:disjoint-geodesics}
For the finite labeled set $\Lambda$ in \cref{lem:geometric}, there are translations $p_k\in\T^3$ such that the closed geodesics
\[
 \mathcal G_k=p_k+H_k,
 \qquad k\in\Lambda,
\]
are pairwise disjoint.  Consequently, they admit pairwise disjoint tubular neighborhoods of a common positive radius.
\end{lemma}

\begin{proof}
Choose the translations inductively.  For a fixed previously chosen $\mathcal G_j=p_j+H_j$, the translations $p$ for which $(p+H_k)\cap\mathcal G_j\ne\varnothing$ form the coset
\[
 p_j+H_j-H_k.
\]
This is a rational subtorus of dimension at most two and therefore has zero Haar measure in $\T^3$.  A finite union of such exceptional sets cannot fill $\T^3$, so the next translation can be chosen outside it.  The resulting curves are compact and pairwise disjoint; since $\Lambda$ is finite, their mutual distance is positive.
\end{proof}

Fix the translated geodesics from \cref{lem:disjoint-geodesics}.  Choose once and for all a number $r_*>0$ such that the closed $2r_*$-tubes around the $\mathcal G_k$ are pairwise disjoint and the normal-coordinate maps below are embeddings.  The number $r_*$ is a fixed geometric parameter depending only on the finite family of tubes; it is independent of $L$, of the later oscillation frequency $\gamma$, and of the small seed amplitude used in the iteration.

Let
\[
 \widehat k:=\frac{\bar k}{|\bar k|},
 \qquad
 \lambda_k:=|\bar k|,
\]
and choose an orthonormal basis $n_{k,1},n_{k,2}$ of $k^\perp=\bar k^\perp$.  The $2r_*$-tube around $\mathcal G_k$ is parametrized by
\begin{equation}\label{eq:tube-coordinates}
 X_k(s,y_1,y_2)
 :=p_k+s\widehat k+y_1n_{k,1}+y_2n_{k,2}\pmod{\Z^3},
 \qquad
 (s,y)\in\R/\lambda_k\Z\times B_{2r_*}.
\end{equation}
The variable $s$ is arclength along the closed geodesic: indeed,
\[
 X_k(s+\lambda_k,y)=X_k(s,y)
\]
because $\lambda_k\widehat k=\bar k\in\Z^3$.  Since $\widehat k,n_{k,1},n_{k,2}$ are orthonormal, $X_k$ is an isometry onto its image.  In these coordinates,
\begin{equation}\label{eq:tube-metric}
 \dd x=\dd s\,\dd y,
 \qquad
 \Delta=\partial_s^2+\Delta_y.
\end{equation}

\subsection{Transverse rescaling and three-dimensional fields}
Rescale the fields in \cref{lem:dipole} to the fixed transverse radius $r_*$ by
\begin{equation}\label{eq:transverse-rescaling}
 \psi_L^{r_*}(y)=r_*^{-1}\psi_L(y/r_*),
 \qquad
 F_L^{r_*}(y)=r_*^{-2}F_L(y/r_*),
 \qquad
 Q_L^{r_*}(y)=Q_L(y/r_*).
\end{equation}
The exponents are chosen so that both potential identities are preserved.  Directly,
\begin{align}
 \Delta_y\psi_L^{r_*}
 &=r_*^{-3}(\Delta\psi_L)(y/r_*)
 =\diver_yF_L^{r_*},\label{eq:scaled-viscous-potential}\\
 \diver_yQ_L^{r_*}
 &=r_*^{-1}(\diver Q_L)(y/r_*)
 =\psi_L^{r_*}.\label{eq:scaled-velocity-potential}
\end{align}
Moreover,
\begin{equation}\label{eq:scaled-profile-norms}
 \|\psi_L^{r_*}\|_{L^2(\R^2)}=1,
\end{equation}
and, because $r_*$ is fixed,
\begin{equation}\label{eq:scaled-small-norms}
 \begin{split}
 &\|\psi_L^{r_*}\|_{L^1(\R^2)}
 +\|F_L^{r_*}\|_{L^1(\R^2)}
 +\|Q_L^{r_*}\|_{L^2(\R^2)}\\
 &\qquad
 +\|Q_L^{r_*}\|_{L^1(\R^2)}
 +\|\nabla Q_L^{r_*}\|_{L^1(\R^2)}
 \le C_{r_*}L^{-1/2}.
 \end{split}
\end{equation}
The profiles in \cref{lem:dipole} are supported in a fixed compact subset of $B_2$; after rescaling they are supported strictly inside $B_{2r_*}$.

In the tube coordinates \eqref{eq:tube-coordinates}, define fields independent of the axial variable $s$ by
\begin{align}
 \phi_{k,L}(s,y)&=\lambda_k^{-1/2}\psi_L^{r_*}(y),\label{eq:phi-k}\\
 f_{k,L}(s,y)&=\lambda_k^{-1/2}F_L^{r_*}(y),\label{eq:f-k}\\
 q_{k,L}(s,y)&=\lambda_k^{-1/2}Q_L^{r_*}(y).\label{eq:q-k}
\end{align}
A two-dimensional vector $c=(c_1,c_2)$ is identified with the normal vector
\[
 c_1n_{k,1}+c_2n_{k,2}\in k^\perp.
\]
Thus $f_{k,L}$ and $q_{k,L}$ have no axial component and take values in $k^\perp$.  Extend $\phi_{k,L}$, $f_{k,L}$, and $q_{k,L}$ by zero outside their tube.  Their supports stay a positive distance from the tube boundary, so these extensions are smooth periodic fields on $\T^3$.

The flat coordinates give the identities used below.  Since the fields are independent of $s$ and the vector fields $f_{k,L},q_{k,L}$ are purely normal,
\begin{equation}\label{eq:transverse-identities}
 \diver f_{k,L}=\Delta\phi_{k,L},
 \qquad
 \diver q_{k,L}=\phi_{k,L},
\end{equation}
and
\begin{equation}\label{eq:axial-independence}
 k\cdot\nabla\phi_{k,L}=0,
 \qquad
 (k\cdot\nabla)f_{k,L}=0,
 \qquad
 (k\cdot\nabla)q_{k,L}=0.
\end{equation}
Indeed, the three-dimensional divergence of a normal field independent of $s$ is its two-dimensional divergence in $y$, while \eqref{eq:tube-metric} gives $\Delta\phi_{k,L}=\Delta_y\phi_{k,L}$.

Define
\begin{equation}\label{eq:mikado-definitions}
 W_{k,L}=\phi_{k,L}k,
 \qquad
 S_{k,L}=k\otimes f_{k,L}+f_{k,L}\otimes k,
 \qquad
 \Omega_{k,L}=k\otimes q_{k,L}-q_{k,L}\otimes k.
\end{equation}

\begin{proposition}[Logarithmic Mikado family]\label{prop:mikado}
For every $L\ge L_0$ and $k\in\Lambda$, the fields in \eqref{eq:mikado-definitions} are smooth periodic fields satisfying
\begin{align}
 &\diver W_{k,L}=0,
 \qquad \diver(W_{k,L}\otimes W_{k,L})=0,\label{eq:W-Euler}\\
 &\int_{\T^3}W_{k,L}\,\dd x=0,
 \qquad
 \int_{\T^3}W_{k,L}\otimes W_{k,L}\,\dd x=k\otimes k,\label{eq:W-averages}\\
 &S_{k,L}=S_{k,L}^T,
 \quad \operatorname{tr}S_{k,L}=0,
 \quad \diver S_{k,L}=\Delta W_{k,L},\label{eq:S-properties}\\
 &\Omega_{k,L}^T=-\Omega_{k,L},
 \quad \diver\Omega_{k,L}=W_{k,L}.\label{eq:Omega-properties}
\end{align}
For different $k$, the supports of $W_{k,L}$, $S_{k,L}$, and $\Omega_{k,L}$ lie in pairwise disjoint tubes.  Moreover,
\begin{equation}\label{eq:mikado-norms}
 \|W_{k,L}\|_{L^2(\T^3)}=|k|,
\end{equation}
while
\begin{equation}\label{eq:mikado-small-norms}
 \begin{split}
 &\|W_{k,L}\|_{L^1(\T^3)}+\|S_{k,L}\|_{L^1(\T^3)}
 +\|\Omega_{k,L}\|_{L^2(\T^3)}\\
 &\qquad
 +\|\Omega_{k,L}\|_{L^1(\T^3)}
 +\|\nabla\Omega_{k,L}\|_{L^1(\T^3)}
 \le C_\Lambda L^{-1/2}.
 \end{split}
\end{equation}
Finally,
\begin{equation}\label{eq:T-def}
 T_{k,L}=W_{k,L}\otimes W_{k,L}-k\otimes k
\end{equation}
satisfies
\begin{equation}\label{eq:T-properties}
 T_{k,L}=T_{k,L}^T,
 \qquad \int_{\T^3}T_{k,L}\,\dd x=0,
 \qquad \diver T_{k,L}=0,
 \qquad \|T_{k,L}\|_{L^1(\T^3)}\le C_\Lambda.
\end{equation}
\end{proposition}

\begin{proof}
Since $W_{k,L}=\phi_{k,L}k$ and $k\cdot\nabla\phi_{k,L}=0$,
\[
 \diver W_{k,L}=k\cdot\nabla\phi_{k,L}=0
\]
and
\[
 \diver(W_{k,L}\otimes W_{k,L})
 =k\,(k\cdot\nabla)(\phi_{k,L}^2)=0.
\]
The factor $\lambda_k^{-1/2}$ is chosen so that
\begin{equation}\label{eq:phi-normalization}
 \begin{split}
 \int_{\T^3}\phi_{k,L}^2\,\dd x
 &=\int_0^{\lambda_k}\lambda_k^{-1}\,\dd s
   \int_{\R^2}|\psi_L^{r_*}(y)|^2\,\dd y\\
 &=1.
 \end{split}
\end{equation}
Also,
\[
 \int_{\R^2}\psi_L^{r_*}(y)\,\dd y
 =r_*\int_{\R^2}\psi_L(y)\,\dd y=0.
\]
Therefore \eqref{eq:W-averages} holds, and \eqref{eq:phi-normalization} also gives \eqref{eq:mikado-norms}.

The tensor $S_{k,L}$ is symmetric by definition.  Since $f_{k,L}\in k^\perp$,
\[
 \operatorname{tr}S_{k,L}=2k\cdot f_{k,L}=0.
\]
Using \eqref{eq:transverse-identities} and \eqref{eq:axial-independence},
\[
 \diver S_{k,L}
 =k\,\diver f_{k,L}+(k\cdot\nabla)f_{k,L}
 =k\Delta\phi_{k,L}
 =\Delta W_{k,L}.
\]
Similarly, $\Omega_{k,L}$ is antisymmetric and
\[
 \diver\Omega_{k,L}
 =k\,\diver q_{k,L}-(k\cdot\nabla)q_{k,L}
 =k\phi_{k,L}
 =W_{k,L}.
\]

All three fields associated with $k$ are supported in the fixed tube around $\mathcal G_k$, hence supports belonging to distinct labels are disjoint.  The tube coordinates give
\begin{align*}
 \|W_{k,L}\|_{L^1(\T^3)}
 &=|k|\lambda_k^{1/2}\|\psi_L^{r_*}\|_{L^1(\R^2)},\\
 \|S_{k,L}\|_{L^1(\T^3)}
 &\le2|k|\lambda_k^{1/2}\|F_L^{r_*}\|_{L^1(\R^2)},\\
 \|\Omega_{k,L}\|_{L^2(\T^3)}
 &\le2|k|\|Q_L^{r_*}\|_{L^2(\R^2)},\\
 \|\Omega_{k,L}\|_{L^1(\T^3)}
 &\le2|k|\lambda_k^{1/2}\|Q_L^{r_*}\|_{L^1(\R^2)},\\
 \|\nabla\Omega_{k,L}\|_{L^1(\T^3)}
 &\le2|k|\lambda_k^{1/2}\|\nabla Q_L^{r_*}\|_{L^1(\R^2)}.
\end{align*}
Because $\Lambda$ and $r_*$ are fixed, \eqref{eq:scaled-small-norms} proves \eqref{eq:mikado-small-norms}.

Finally, $T_{k,L}$ is symmetric, its mean vanishes by \eqref{eq:W-averages}, and its divergence vanishes by \eqref{eq:W-Euler}.  Since $|\T^3|=1$,
\[
 \|T_{k,L}\|_{L^1(\T^3)}
 \le\|W_{k,L}\otimes W_{k,L}\|_{L^1(\T^3)}+|k\otimes k|
 =2|k|^2,
\]
which proves \eqref{eq:T-properties} after enlarging $C_\Lambda$.
\end{proof}

If $\gamma\in\N$, the map $x\mapsto\gamma x$ is a measure-preserving covering of $\T^3$.  Consequently,
\begin{equation}\label{eq:covering-norms}
 \|G(\gamma\,\cdot)\|_{L^p(\T^3)}=\|G\|_{L^p(\T^3)},
 \qquad
 \int_{\T^3}G(\gamma x)\,\dd x=\int_{\T^3}G(x)\,\dd x
\end{equation}
for every periodic $G$ and $1\le p<\infty$.  Inverse images of disjoint sets remain disjoint, so the supports of $W_{k,L}(\gamma x)$, $S_{k,L}(\gamma x)$, and $\Omega_{k,L}(\gamma x)$ remain pairwise disjoint as $k$ varies.

In the whole-space sections below, the notation $G(\gamma x)$ means the $\Z^3$-periodic lift of $G$ evaluated at $x\in\R^3$.  Such a lift is not integrable over all of $\R^3$ by itself.  It always appears multiplied by a compactly supported slow amplitude; norms of those products are whole-space norms and are controlled by \cref{lem:periodic-cells}.

\section{Compactly supported divergence-free localization}\label{sec:localization}

From this section onward, $v,w,z,u,R$ and all modulated stresses are fields on $\R^3$; their displayed norms are therefore whole-space norms.  Norms of the unmodulated periodic profiles $W_{k,L}$, $S_{k,L}$, $\Omega_{k,L}$, and $T_{k,L}$ continue to be written with the domain $\T^3$ explicitly.

Let $U\subset\R^3$ be the fixed ball in the Reynolds system, and let
\[
 a_k\in C_c^\infty(U),\qquad k\in\Lambda.
\]
Define the principal perturbation
\begin{equation}\label{eq:v-def}
 v(x)=\sum_{k\in\Lambda}a_k(x)W_{k,L}(\gamma x).
\end{equation}
We localize it exactly by means of the antisymmetric potentials:
\begin{equation}\label{eq:w-def}
 w_i(x)=\sum_{k\in\Lambda}\gamma^{-1}\partial_j
 \left(a_k(x)(\Omega_{k,L})_{ij}(\gamma x)\right).
\end{equation}
Since each $a_k\Omega_{k,L}(\gamma\cdot)$ is antisymmetric,
\begin{equation}\label{eq:w-divfree}
 \diver w=\gamma^{-1}\partial_i\partial_j
 \sum_k a_k(\Omega_{k,L})_{ij}(\gamma x)=0.
\end{equation}
Moreover, $w\in C_c^\infty(U;\R^3)$.  Expanding \eqref{eq:w-def} gives
\begin{equation}\label{eq:w-v-z}
 w=v+z,
 \qquad
 z_i=\sum_{k\in\Lambda}\gamma^{-1}(\partial_j a_k)
 (\Omega_{k,L})_{ij}(\gamma x).
\end{equation}

\begin{lemma}[Localized velocity estimates]\label{lem:localized-velocity}
Let $a=(a_k)_{k\in\Lambda}\subset C_c^\infty(U)$.  Uniformly in $L\ge L_0$ and $\gamma\in\N$,
\begin{align}
 \|v\|_{L^1(\R^3)}&\le C_aL^{-1/2},\label{eq:v-L1}\\
 \|z\|_{L^1(\R^3)}+\|z\|_{L^2(\R^3)}&\le C_a(\gamma\sqrt L)^{-1},\label{eq:z-L12}\\
 \|\nabla z\|_{L^1(\R^3)}&\le C_aL^{-1/2},\label{eq:z-W11}
\end{align}
and
\begin{equation}\label{eq:v-L2}
 \|v\|_{L^2(\R^3)}
 \le \left(\sum_{k\in\Lambda}|k|^2\|a_k\|_{L^2(\R^3)}^2\right)^{1/2}
 +C_a\gamma^{-1/2}.
\end{equation}
Consequently,
\begin{equation}\label{eq:w-L2}
 \|w\|_{L^2(\R^3)}
 \le \left(\sum_{k\in\Lambda}|k|^2\|a_k\|_{L^2(\R^3)}^2\right)^{1/2}
 +C_a\gamma^{-1/2}+C_a(\gamma\sqrt L)^{-1}.
\end{equation}
For every fixed $\varphi\in C_c^\infty(\R^3;\R^3)$,
\begin{equation}\label{eq:test-protection}
 \left|\int_{\R^3}w\cdot\varphi\,\dd x\right|
 \le C_{a,\varphi}(\gamma\sqrt L)^{-1}.
\end{equation}
In \eqref{eq:v-L1}--\eqref{eq:w-L2}, one may take
\[
 C_a=C\bigl(U,\Lambda,\AmpNorm{2}{a}\bigr).
\]
In \eqref{eq:test-protection}, one may take
\[
 C_{a,\varphi}=C\bigl(U,\Lambda,\AmpNorm{0}{a},\|\varphi\|_{C^1}\bigr).
\]
These constants are independent of positive derivative norms of the profiles.
\end{lemma}

\begin{proof}
The estimates \eqref{eq:v-L1} and \eqref{eq:z-L12} follow from \cref{lem:periodic-cells,eq:mikado-small-norms}.  Differentiating the corrector gives
\[
 \partial_\ell z_i
 =\sum_k\gamma^{-1}(\partial_{\ell j}a_k)(\Omega_{k,L})_{ij}(\gamma x)
 +\sum_k(\partial_ja_k)(\partial_{y_\ell}\Omega_{k,L})_{ij}(\gamma x),
\]
which proves \eqref{eq:z-W11} using the periodic bounds $\|\Omega_{k,L}\|_{L^1(\T^3)}$ and $\|\nabla\Omega_{k,L}\|_{L^1(\T^3)}$.

Because the velocity supports are disjoint in $k$,
\[
 |v(x)|^2=\sum_k a_k(x)^2|W_{k,L}(\gamma x)|^2.
\]
Apply \cref{eq:cell-average} to each summand and use \eqref{eq:mikado-norms}.  The error in the square of the norm is $O_a(\gamma^{-1})$, and \eqref{eq:v-L2} follows from $\sqrt{A+B}\le\sqrt A+\sqrt B$.  Together with \eqref{eq:z-L12}, this gives \eqref{eq:w-L2}.

Finally, integrate \eqref{eq:w-def} by parts:
\[
 \int w_i\varphi_i
 =-\gamma^{-1}\sum_k\int a_k(x)(\Omega_{k,L})_{ij}(\gamma x)
 \partial_j\varphi_i(x)\,\dd x.
\]
The periodic cell estimate and the bound $\|\Omega_{k,L}\|_{L^1(\T^3)}\lesssim L^{-1/2}$ from \eqref{eq:mikado-small-norms} yield \eqref{eq:test-protection}.
\end{proof}

\section{The modulated viscosity identity}\label{sec:viscosity}

The crucial point is to represent $\Delta(aW(\gamma x))$ as the divergence of a compactly supported symmetric tensor without estimating $\nabla W$.

Fix one direction $k\in\Lambda$ and suppress the indices $k,L$.  Thus
\[
 W=\psi k,
 \qquad k\cdot\nabla_y\psi=0,
 \qquad \diver S=\Delta W,
 \qquad S=S^T,
 \qquad \operatorname{tr}S=0.
\]
Let
\begin{equation}\label{eq:Pk-b}
 P_k=I-\widehat k\otimes\widehat k,
 \qquad \widehat k=\frac{k}{|k|},
 \qquad \beta=P_k\nabla a.
\end{equation}
Define
\begin{equation}\label{eq:C-mod}
 C[a,W](x)=2\psi(\gamma x)\bigl(k\otimes \beta(x)+\beta(x)\otimes k\bigr).
\end{equation}
It is symmetric and trace-free because $\beta\perp k$.

\begin{lemma}[Modulated viscosity]\label{lem:modulated-viscosity}
For every $a\in C_c^\infty(U)$ there exists $\calE_\gamma[a,W,S]\in C_c^\infty(U;\Sym_3)$ such that
\begin{equation}\label{eq:viscosity-exact}
 \diver\calE_\gamma[a,W,S]
 =\Delta\bigl(a(x)W(\gamma x)\bigr)
\end{equation}
and
\begin{equation}\label{eq:viscosity-bound}
 \|\calE_\gamma[a,W,S]\|_{L^1(\R^3)}
 \le C\bigl(U,\|a\|_{C^2}\bigr)\bigl(\gamma\|S\|_{L^1(\T^3)}+\|W\|_{L^1(\T^3)}\bigr).
\end{equation}
In particular, for the logarithmic profiles,
\begin{equation}\label{eq:viscosity-log}
 \|\calE_\gamma[a,W_{k,L},S_{k,L}]\|_{L^1(\R^3)}
 \le C\bigl(U,\Lambda,\|a\|_{C^2}\bigr)\frac\gamma{\sqrt L}.
\end{equation}
The constants are independent of $\nabla W$, $\nabla S$, and the inner scale $e^{-L}$.
\end{lemma}

\begin{proof}
Set
\begin{equation}\label{eq:E0}
 E_0=a(x)\gamma S(\gamma x)+C[a,W](x).
\end{equation}
We compute the two divergences.  First,
\begin{equation}\label{eq:div-aS}
 \diver(a\gamma S(\gamma x))
 =a\gamma^2\Delta W(\gamma x)+\gamma S(\gamma x)\nabla a.
\end{equation}
Second, using $k\cdot\nabla_y\psi=0$ and $\beta\cdot\nabla_y\psi=\nabla a\cdot\nabla_y\psi$,
\begin{align}
 \diver C[a,W]
 &=2\gamma\bigl(\nabla a\cdot\nabla_y\psi(\gamma x)\bigr)k\notag\\
 &\quad+2\psi(\gamma x)\bigl(k\,\diver\beta+(k\cdot\nabla)\beta\bigr).
 \label{eq:div-C}
\end{align}
On the other hand,
\begin{align}
 \Delta(aW(\gamma x))
 &=a\gamma^2\Delta W(\gamma x)
 +2\gamma\bigl(\nabla a\cdot\nabla_y\psi(\gamma x)\bigr)k\notag\\
 &\quad +(\Delta a)\psi(\gamma x)k.
 \label{eq:lap-product}
\end{align}
Subtracting \eqref{eq:div-aS} and \eqref{eq:div-C} from \eqref{eq:lap-product}, define
\begin{equation}\label{eq:g-visc}
 \begin{split}
 g={}&(\Delta a)\psi(\gamma x)k
 -\gamma S(\gamma x)\nabla a\\
 &-2\psi(\gamma x)\bigl(k\,\diver\beta+(k\cdot\nabla)\beta\bigr).
 \end{split}
\end{equation}
Then
\begin{equation}\label{eq:lap-divE0-g}
 \Delta(aW(\gamma x))=\diver E_0+g.
\end{equation}
The periodic cell estimate gives
\begin{equation}\label{eq:g-visc-L1}
 \|g\|_{L^1(\R^3)}
 \le C\bigl(U,\|a\|_{C^2}\bigr)\bigl(\gamma\|S\|_{L^1(\T^3)}+\|W\|_{L^1(\T^3)}\bigr).
\end{equation}
No derivative of the periodic profile appears.

It remains to verify the compatibility conditions for $g$.  Put $V=aW(\gamma x)$ and define the compactly supported symmetric tensor
\begin{equation}\label{eq:A-lap}
 A[V]=\nabla V+(\nabla V)^T-(\diver V)I.
\end{equation}
A direct calculation gives $\diver A[V]=\Delta V$.  Since $E_0$ is also symmetric, \eqref{eq:lap-divE0-g} shows
\[
 g=\diver(A[V]-E_0).
\]
Thus \cref{lem:auto-compat} applies.  Set
\begin{equation}\label{eq:E-final}
 \calE_\gamma[a,W,S]=E_0+\Sloc g.
\end{equation}
Then \eqref{eq:viscosity-exact} follows from \eqref{eq:local-right-inverse}.  The estimates for $E_0$, \eqref{eq:g-visc-L1}, and \eqref{eq:local-L1} give \eqref{eq:viscosity-bound}.  Finally use \cref{eq:mikado-small-norms} to obtain \eqref{eq:viscosity-log}.
\end{proof}

The divergence corrector $z$ in \eqref{eq:w-v-z} has a simpler viscous stress.

\begin{lemma}[Corrector viscosity]\label{lem:corrector-viscosity}
Let
\begin{equation}\label{eq:Ez}
 E_z=\nabla z+(\nabla z)^T-(\diver z)I.
\end{equation}
Then $E_z\in C_c^\infty(U;\Sym_3)$,
\begin{equation}\label{eq:Ez-div}
 \diver E_z=\Delta z,
\end{equation}
and
\begin{equation}\label{eq:Ez-bound}
 \|E_z\|_{L^1(\R^3)}\le C\bigl(U,\Lambda,\AmpNorm{2}{a}\bigr)L^{-1/2}.
\end{equation}
\end{lemma}

\begin{proof}
The identity follows by expanding the divergence.  The estimate is a direct consequence of \cref{eq:z-W11}.
\end{proof}

Combining all directions, define
\begin{equation}\label{eq:Evis}
 E_{\mathrm{vis}}
 =\sum_{k\in\Lambda}\calE_\gamma[a_k,W_{k,L},S_{k,L}]+E_z.
\end{equation}
Then
\begin{equation}\label{eq:Evis-final}
 \diver E_{\mathrm{vis}}=\Delta w,
 \qquad
 \|E_{\mathrm{vis}}\|_{L^1(\R^3)}
 \le C\bigl(U,\Lambda,\AmpNorm{2}{a}\bigr)\frac\gamma{\sqrt L}.
\end{equation}

\section{A compactly supported oscillation stress}\label{sec:oscillation}

The oscillation error has the form
\[
 \diver\bigl(b(x)T(\gamma x)\bigr),
\]
where $b\in C_c^\infty(U)$ and $T\in C^\infty(\T^3;\Sym_3)$ satisfies
\begin{equation}\label{eq:T-assumptions}
 \int_{\T^3}T=0,
 \qquad \diver_yT=0.
\end{equation}
The next construction gains one power of $\gamma$ without any derivative of $T$.

\begin{lemma}[Localized oscillation anti-divergence]\label{lem:oscillation}
Under \eqref{eq:T-assumptions}, there exists
\[
 \calO_\gamma[b,T]\in C_c^\infty(U;\Sym_3)
\]
such that
\begin{equation}\label{eq:osc-div}
 \diver\calO_\gamma[b,T]
 =\diver\bigl(b(x)T(\gamma x)\bigr)
\end{equation}
and
\begin{equation}\label{eq:osc-bound}
 \|\calO_\gamma[b,T]\|_{L^1(\R^3)}
 \le C\bigl(U,\|b\|_{C^2}\bigr)\gamma^{-1}\|T\|_{L^1(\T^3)}.
\end{equation}
The constant is independent of derivative norms of $T$; the periodic operator $\Rop$ is applied separately to each row of $T$.
\end{lemma}

\begin{proof}
For each $\ell\in\{1,2,3\}$, let $T_{\ell\cdot}$ denote the $\ell$th row of $T$ and define
\[
 Q^\ell=\Rop(T_{\ell\cdot}).
\]
Then $Q^\ell$ is symmetric and trace-free,
\begin{equation}\label{eq:Qell-div}
 \partial_{y_j}Q^\ell_{ij}=T_{\ell i}=T_{i\ell},
\end{equation}
and, by \eqref{eq:Rtorus-L1},
\begin{equation}\label{eq:Qell-L1}
 \|Q^\ell\|_{L^1(\T^3)}\le C\|T\|_{L^1(\T^3)}.
\end{equation}
Set
\begin{equation}\label{eq:B0-osc}
 (B_0)_{ij}(x)
 =\gamma^{-1}(\partial_\ell b)(x)Q^\ell_{ij}(\gamma x).
\end{equation}
Using \eqref{eq:Qell-div}, symmetry of $T$, and $\diver_yT=0$, we obtain
\begin{align}
 (\diver B_0)_i
 &=\gamma^{-1}(\partial_{j\ell}b)Q^\ell_{ij}(\gamma x)
 +(\partial_\ell b)T_{i\ell}(\gamma x)\notag\\
 &=r_i+\bigl(\diver(bT(\gamma x))\bigr)_i,
 \label{eq:B0-div}
\end{align}
where
\begin{equation}\label{eq:r-osc}
 r_i=\gamma^{-1}(\partial_{j\ell}b)Q^\ell_{ij}(\gamma x).
\end{equation}
The periodic cell estimate and \eqref{eq:Qell-L1} give
\begin{equation}\label{eq:B0-r-bound}
 \|B_0\|_{L^1(\R^3)}+\|r\|_{L^1(\R^3)}
 \le C\bigl(U,\|b\|_{C^2}\bigr)\gamma^{-1}\|T\|_{L^1(\T^3)}.
\end{equation}
Moreover,
\[
 r=\diver\bigl(B_0-bT(\gamma x)\bigr),
\]
and the tensor in parentheses is compactly supported and symmetric.  Hence $r$ satisfies the force and torque conditions by \cref{lem:auto-compat}.  Define
\begin{equation}\label{eq:O-def}
 \calO_\gamma[b,T]=B_0-\Sloc r.
\end{equation}
Then \eqref{eq:osc-div} follows from \eqref{eq:B0-div}, and \eqref{eq:osc-bound} follows from \eqref{eq:B0-r-bound} and \eqref{eq:local-L1}.
\end{proof}

\section{The compactly supported perturbation proposition}\label{sec:perturbation}

We now combine the preceding ingredients.  The strengthened form below records, in addition to the $L^2$ increment and the new Reynolds stress, independent $L^1$, $H^{-1}$, pressure, and finitely many pairing controls. 

\begin{proposition}[Endpoint perturbation in $\R^3$]\label{prop:perturbation}
Let $(u,p,R)$ be a smooth solution of \eqref{eq:NSR} with
\[
 u\in C_{c,\sigma}^\infty(U;\R^3),
 \qquad p\in C_c^\infty(U),
 \qquad R\in C_c^\infty(U;\Szero),
\]
where $U\subset\R^3$ is a ball.  Fix positive numbers
\[
 \delta,\quad \eta_2,\quad \eta_1,\quad \eta_{-1},\quad \zeta,
\]
and finitely many test fields $\varphi_1,\dots,\varphi_N\in C_c^\infty(\R^3;\R^3)$ with tolerances $\kappa_1,\dots,\kappa_N>0$.  Then there is another smooth compactly supported solution $(\overline u,\overline p,\overline R)$ of \eqref{eq:NSR} in $U$ such that, writing $w=\overline u-u$,
\begin{align}
 \|\overline R\|_{L^1(\R^3)}&<\delta,\label{eq:prop-new-stress}\\
 \|w\|_{L^2(\R^3)}&\le \sqrt{12}\,\|R\|_{L^1(\R^3)}^{1/2}+\eta_2,\label{eq:prop-L2}\\
 \|w\|_{L^1(\R^3)}&<\eta_1,\label{eq:prop-L1}\\
 \|w\|_{H^{-1}(\R^3)}&<\eta_{-1},\label{eq:prop-Hminus1}\\
 \left|\int_{\R^3}w\cdot\varphi_j\,\dd x\right|&<\kappa_j
 \qquad(1\le j\le N),\label{eq:prop-pairing}\\
 \|\overline p-p\|_{L^1(\R^3)}&\le 4\|R\|_{L^1(\R^3)}+\zeta.\label{eq:prop-pressure}
\end{align}
All constants implicit in the proof are independent of the current parameters $\gamma$ and $L$.
\end{proposition}

\begin{proof}
\medskip\noindent\textbf{Step 1: compact amplitude decomposition.}
Choose $\theta\in C_c^\infty(U)$ such that $0\le\theta\le1$ and $\theta=1$ on a neighborhood of $\supp R$.  The transition region of $\theta$ can be chosen disjoint from $\supp R$.  Let $\sigma>0$ be a small auxiliary parameter and define
\begin{equation}\label{eq:rho-def}
 \rho=4\bigl(\sigma^2\theta^4+|R|^2\bigr)^{1/2}.
\end{equation}
On the region where $R$ may be nonzero, $\theta=1$ and $\rho>0$; on the transition region, $R=0$ and
\[
 \rho^{1/2}=2\sigma^{1/2}\theta.
\]
Thus $\rho^{1/2}\in C_c^\infty(U)$.  Define
\[
 M=I-\rho^{-1}R
\]
where $\rho>0$, and set $M=I$ where $\rho=0$.  This extension is smooth, and
\[
 |M-I|\le\frac14.
\]
By \cref{lem:geometric}, set
\begin{equation}\label{eq:a-def}
 a_k=\rho^{1/2}\Gamma_k(M).
\end{equation}
Then $a_k\in C_c^\infty(U)$ and
\begin{equation}\label{eq:stress-decomp}
 \rho I-R=\sum_{k\in\Lambda}a_k^2k\otimes k.
\end{equation}
Taking the trace in \eqref{eq:geometric} and using $\operatorname{tr}R=0$ gives
\begin{equation}\label{eq:amplitude-trace}
 \sum_{k\in\Lambda}|k|^2a_k^2=3\rho.
\end{equation}
Since
\[
 \rho\le4\sigma\theta^2+4|R|,
\]
we obtain
\begin{equation}\label{eq:amplitude-energy}
 \sum_{k\in\Lambda}|k|^2\|a_k\|_{L^2(\R^3)}^2
 \le12\|R\|_{L^1(\R^3)}+12\sigma\|\theta\|_{L^2(\R^3)}^2.
\end{equation}

\medskip\noindent\textbf{Step 2: velocity perturbation and low-norm controls.}
Choose an integer $\gamma\ge1$ and $L\ge L_0$, and define $v,w,z$ by \eqref{eq:v-def}, \eqref{eq:w-def}, and \eqref{eq:w-v-z}.  By disjointness of the profiles,
\begin{equation}\label{eq:v-tensor}
 v\otimes v
 =\sum_{k\in\Lambda}a_k^2W_{k,L}(\gamma x)\otimes W_{k,L}(\gamma x).
\end{equation}
Therefore, with $T_{k,L}$ from \eqref{eq:T-def},
\begin{equation}\label{eq:main-cancellation}
 R+v\otimes v
 =\rho I+\sum_{k\in\Lambda}a_k^2T_{k,L}(\gamma x).
\end{equation}
The velocity estimate \eqref{eq:w-L2} and \eqref{eq:amplitude-energy} give
\begin{equation}\label{eq:w-L2-current}
 \begin{split}
 \|w\|_{L^2(\R^3)}
 \le{}&\sqrt{12}\,\|R\|_{L^1(\R^3)}^{1/2}
 +\sqrt{12}\,\|\theta\|_{L^2(\R^3)}\sigma^{1/2}\\
 &+C_{R,\sigma,\theta,U}\gamma^{-1/2}
 +C_{R,\sigma,\theta,U}(\gamma\sqrt L)^{-1}.
 \end{split}
\end{equation}
Moreover,
\begin{equation}\label{eq:w-L1-current}
 \|w\|_{L^1(\R^3)}\le C_{R,\sigma,\theta,U}L^{-1/2},
\end{equation}
and \eqref{eq:test-protection} gives, for each $j$,
\begin{equation}\label{eq:pair-current}
 \left|\int_{\R^3}w\cdot\varphi_j\,\dd x\right|
 \le C_{R,\sigma,\theta,U,\varphi_j}(\gamma\sqrt L)^{-1}.
\end{equation}

The $H^{-1}$ estimate is implicit in the antisymmetric localization.  Put
\[
 B_{ij}(x)=\sum_{k\in\Lambda}a_k(x)(\Omega_{k,L})_{ij}(\gamma x).
\]
Then \eqref{eq:w-def} reads $w_i=\gamma^{-1}\partial_jB_{ij}$.  The periodic cell estimate and \eqref{eq:mikado-small-norms} imply
\[
 \|B\|_{L^2(\R^3)}\le C_{R,\sigma,\theta,U}L^{-1/2}.
\]
Consequently, for every $\phi\in H^1(\R^3;\R^3)$,
\[
 \left|\int_{\R^3}w\cdot\phi\,\dd x\right|
 \le\gamma^{-1}\|B\|_2\|\nabla\phi\|_2
 \le C_{R,\sigma,\theta,U}(\gamma\sqrt L)^{-1}\|\phi\|_{H^1},
\]
and hence
\begin{equation}\label{eq:w-Hminus1-current}
 \|w\|_{H^{-1}(\R^3)}
 \le C_{R,\sigma,\theta,U}(\gamma\sqrt L)^{-1}.
\end{equation}

\medskip\noindent\textbf{Step 3: viscosity and oscillation stresses.}
Let $E_{\mathrm{vis}}$ be defined by \eqref{eq:Evis}.  Then
\begin{equation}\label{eq:vis-current}
 \diver E_{\mathrm{vis}}=\Delta w,
 \qquad
 \|E_{\mathrm{vis}}\|_{L^1(\R^3)}
 \le C_{R,\sigma,\theta,U}\frac\gamma{\sqrt L}.
\end{equation}
For every $k$, apply \cref{lem:oscillation} with $b=a_k^2$ and $T=T_{k,L}$, and set
\begin{equation}\label{eq:Rosc-sum}
 R_{\mathrm{osc}}
 =\sum_{k\in\Lambda}\calO_\gamma[a_k^2,T_{k,L}].
\end{equation}
By \cref{eq:T-properties},
\begin{equation}\label{eq:Rosc-properties}
 \diver R_{\mathrm{osc}}
 =\diver\sum_k a_k^2T_{k,L}(\gamma x),
 \qquad
 \|R_{\mathrm{osc}}\|_{L^1(\R^3)}
 \le C_{R,\sigma,\theta,U}\gamma^{-1}.
\end{equation}

\medskip\noindent\textbf{Step 4: the new Reynolds equation.}
Define the symmetric compactly supported tensor
\begin{equation}\label{eq:A-new}
 \begin{split}
 A={}&-E_{\mathrm{vis}}+R_{\mathrm{osc}}
 +u\otimes w+w\otimes u\\
 &+w\otimes w-v\otimes v.
 \end{split}
\end{equation}
Set
\begin{equation}\label{eq:new-fields}
 \overline u=u+w,
 \qquad
 \overline R=\dev A,
 \qquad
 \overline p=p-\rho-\frac13\operatorname{tr}A.
\end{equation}
Using the old Reynolds equation, \eqref{eq:main-cancellation}, \eqref{eq:vis-current}, and \eqref{eq:Rosc-properties}, we obtain
\[
 -\Delta\overline u+\diver(\overline u\otimes\overline u)+\nabla(p-\rho)
 =\diver A.
\]
The pressure correction in \eqref{eq:new-fields} therefore gives
\begin{equation}\label{eq:new-NSR}
 -\Delta\overline u+\diver(\overline u\otimes\overline u)+\nabla\overline p
 =\diver\overline R.
\end{equation}
All fields are smooth and compactly supported in $U$, and $\diver\overline u=0$.

\medskip\noindent\textbf{Step 5: estimate of the full tensor $A$.}
The viscosity and oscillation terms satisfy the bounds above.  By \eqref{eq:w-L1-current},
\begin{equation}\label{eq:linear-error}
 \|u\otimes w+w\otimes u\|_1
 \le2\|u\|_\infty\|w\|_1
 \le C_{u,R,\sigma,\theta,U}L^{-1/2}.
\end{equation}
Since $w=v+z$,
\begin{align}
 \|w\otimes w-v\otimes v\|_1
 &\le2\|v\|_2\|z\|_2+\|z\|_2^2\notag\\
 &\le C_{R,\sigma,\theta,U}(\gamma\sqrt L)^{-1}.
 \label{eq:corrector-quadratic}
\end{align}
Enlarging the constant and using $\gamma\ge1$, we find
\begin{equation}\label{eq:A-estimate}
 \|A\|_{L^1(\R^3)}
 \le C_{u,R,\sigma,\theta,U}
 \left(\frac1\gamma+\frac\gamma{\sqrt L}\right).
\end{equation}
Since $\overline R=\dev A$, \eqref{eq:dev-contraction} gives $\|\overline R\|_1\le\|A\|_1$.

\medskip\noindent\textbf{Step 6: parameter choice and pressure estimate.}
First choose $\sigma>0$ so small that
\begin{equation}\label{eq:sigma-choice}
 \sqrt{12}\,\|\theta\|_{L^2}\sigma^{1/2}<\frac{\eta_2}{3},
 \qquad
 4\sigma\|\theta\|_{L^2}^2<\frac\zeta2.
\end{equation}
Once $\sigma$ is fixed, all amplitude $C^2$ norms are finite and depend only on the old data, $\theta$, and $\sigma$.  Choose the integer $\gamma$ sufficiently large that
\begin{equation}\label{eq:gamma-choice}
 C_{R,\sigma,\theta,U}\gamma^{-1/2}<\frac{\eta_2}{3},
 \qquad
 C_{u,R,\sigma,\theta,U}\gamma^{-1}
 <\min\left\{\frac\delta2,\frac{\sqrt3\,\zeta}{4}\right\}.
\end{equation}
Finally choose $L$ so large that
\begin{equation}\label{eq:L-choice}
 \begin{split}
 C_{R,\sigma,\theta,U}(\gamma\sqrt L)^{-1}&<\frac{\eta_2}{3},\\
 C_{R,\sigma,\theta,U}L^{-1/2}&<\eta_1,\\
 C_{R,\sigma,\theta,U}(\gamma\sqrt L)^{-1}&<\eta_{-1},\\
 C_{R,\sigma,\theta,U,\varphi_j}(\gamma\sqrt L)^{-1}&<\kappa_j
 \quad(1\le j\le N),\\
 C_{u,R,\sigma,\theta,U}\gamma L^{-1/2}
 &<\min\left\{\frac\delta2,\frac{\sqrt3\,\zeta}{4}\right\}.
 \end{split}
\end{equation}
The choices are compatible because $L$ has no upper bound.  The preceding estimates give \eqref{eq:prop-new-stress}--\eqref{eq:prop-pairing}.  Finally,
\[
 \|\rho\|_1\le4\sigma\|\theta\|_2^2+4\|R\|_1,
 \qquad
 |\operatorname{tr}A|\le\sqrt3\,|A|.
\]
Together with \eqref{eq:new-fields}, \eqref{eq:sigma-choice}, and \eqref{eq:A-estimate}, this yields
\[
 \|\overline p-p\|_1
 \le4\|R\|_1+4\sigma\|\theta\|_2^2+\frac1{\sqrt3}\|A\|_1
 \le4\|R\|_1+\zeta,
\]
which is \eqref{eq:prop-pressure}.
\end{proof}

\section{Iteration and proof of the main theorem}\label{sec:iteration}

Choose an open ball $U\Subset B$.  We run the whole iteration in this fixed inner ball.  Choose a nonzero field
\[
 U_0\in C_{c,\sigma}^\infty(U;\R^3)
\]
and, for a parameter $\tau>0$, set
\begin{equation}\label{eq:seed-u}
 u_0=\tau U_0.
\end{equation}
Define
\begin{equation}\label{eq:seed-A}
 A_0=-\nabla u_0-(\nabla u_0)^T+u_0\otimes u_0,
\end{equation}
and
\begin{equation}\label{eq:seed-R-p}
 R_0=\dev A_0,
 \qquad
 p_0=-\frac13\operatorname{tr}A_0=-\frac13|u_0|^2.
\end{equation}
Since $\diver u_0=0$,
\[
 \diver A_0=-\Delta u_0+\diver(u_0\otimes u_0),
\]
and hence $(u_0,p_0,R_0)$ solves \eqref{eq:NSR}.  Moreover,
\begin{equation}\label{eq:seed-stress-size}
 \|R_0\|_{L^1(\R^3)}\le C_{U_0}(\tau+\tau^2),
 \qquad
 \|p_0\|_{L^1(\R^3)}=\frac{\tau^2}{3}\|U_0\|_{L^2(\R^3)}^2.
\end{equation}

Fix $C_*\ge C_{U_0}$ large enough that
\[
 \|R_0\|_{L^1(\R^3)}\le C_*(\tau+\tau^2),
\]
and define
\begin{equation}\label{eq:delta-sequence}
 \delta_0=C_*(\tau+\tau^2),
 \qquad
 \delta_q=\delta_0 2^{-4q},
 \qquad q\ge0.
\end{equation}
We construct smooth compactly supported Reynolds solutions $(u_q,p_q,R_q)$ satisfying
\begin{equation}\label{eq:iteration-invariants}
 \supp u_q\cup\supp p_q\cup\supp R_q\subset U,
 \qquad
 \|R_q\|_{L^1(\R^3)}\le\delta_q.
\end{equation}
Suppose $(u_q,p_q,R_q)$ has been constructed.  Set
\begin{equation}\label{eq:iteration-targets}
 \eta_q=2^{-q-3}\delta_0^{1/2},
 \qquad
 \kappa_q=2^{-q-3}\tau\|U_0\|_{L^2(\R^3)}^2,
 \qquad
 \zeta_q=\tau 2^{-q-2}.
\end{equation}
Apply \cref{prop:perturbation} with
\[
 \delta=\delta_{q+1},
 \quad \eta_2=\eta_q,
 \quad \eta_1=1,
 \quad \eta_{-1}=1,
 \quad \zeta=\zeta_q,
 \quad N=1,
 \quad \varphi_1=U_0,
 \quad \kappa_1=\kappa_q.
\]
Writing $w_{q+1}=u_{q+1}-u_q$, we obtain
\begin{align}
 \|R_{q+1}\|_{L^1(\R^3)}&<\delta_{q+1},\label{eq:Rq-decay}\\
 \|w_{q+1}\|_{L^2(\R^3)}
 &\le \sqrt{12}\,\delta_q^{1/2}+\eta_q,\label{eq:wq-summable}\\
 \left|\int_{\R^3}w_{q+1}\cdot U_0\,\dd x\right|
 &<\kappa_q,\label{eq:pair-summable}\\
 \|p_{q+1}-p_q\|_{L^1(\R^3)}
 &\le4\delta_q+\zeta_q.\label{eq:pq-summable}
\end{align}
Since
\[
 \sum_{q=0}^\infty\delta_q^{1/2}<\infty,
 \qquad
 \sum_{q=0}^\infty\eta_q<\infty,
\]
there exists $u\in L^2(\R^3)$ such that
\begin{equation}\label{eq:strong-L2}
 u_q\longrightarrow u
 \quad\text{strongly in }L^2(\R^3).
\end{equation}
Likewise, \eqref{eq:pq-summable} gives a function $p\in L^1(\R^3)$ such that
\begin{equation}\label{eq:strong-pressure}
 p_q\longrightarrow p
 \quad\text{strongly in }L^1(\R^3).
\end{equation}
Every $u_q$ and $p_q$ vanishes outside the fixed inner ball $U$.  Hence
\begin{equation}\label{eq:limit-support}
 \esssupp u\cup\esssupp p\subset\overline U\subset B.
\end{equation}
The divergence-free condition passes to the limit.

Strong $L^2$ convergence gives
\begin{equation}\label{eq:quadratic-convergence}
 u_q\otimes u_q\longrightarrow u\otimes u
 \quad\text{strongly in }L^1(\R^3).
\end{equation}
For every $\phi\in C_c^\infty(\R^3;\R^3)$, the Reynolds equation gives
\begin{equation}\label{eq:full-weak-q}
 \int u_q\cdot(-\Delta\phi)
 -\int(u_q\otimes u_q):\nabla\phi
 -\int p_q\,\diver\phi
 =-\int R_q:\nabla\phi.
\end{equation}
The right-hand side tends to zero by \eqref{eq:Rq-decay}; the three terms on the left converge by \eqref{eq:strong-L2}, \eqref{eq:quadratic-convergence}, and \eqref{eq:strong-pressure}.  Thus $(u,p)$ solves \eqref{eq:NS} in distributions.

To prove nontriviality, sum \eqref{eq:pair-summable}:
\[
 \left|\int_{\R^3}(u-u_0)\cdot U_0\,\dd x\right|
 \le\sum_{q=0}^\infty\kappa_q
 =\frac14\tau\|U_0\|_{L^2(\R^3)}^2.
\]
Since
\[
 \int u_0\cdot U_0=\tau\|U_0\|_{L^2(\R^3)}^2,
\]
we obtain
\begin{equation}\label{eq:nonzero-pairing}
 \int u\cdot U_0
 \ge\frac34\tau\|U_0\|_{L^2(\R^3)}^2>0.
\end{equation}
Hence $u\ne0$.

Finally, for $0<\tau\le1$, \eqref{eq:wq-summable} and \eqref{eq:delta-sequence} give
\begin{equation}\label{eq:small-norm}
 \|u\|_{L^2(\R^3)}
 \le C_{U_0}(\tau+\tau^{1/2}).
\end{equation}
Moreover, by \eqref{eq:seed-stress-size}, \eqref{eq:pq-summable}, and the geometric series,
\begin{equation}\label{eq:small-pressure}
 \|p\|_{L^1(\R^3)}
 \le \|p_0\|_1+4\sum_{q=0}^\infty\delta_q+\sum_{q=0}^\infty\zeta_q
 \le C_{U_0}(\tau+\tau^2).
\end{equation}
Choosing $\tau$ sufficiently small gives
\[
 \|u\|_{L^2(\R^3)}+\|p\|_{L^1(\R^3)}<\eps.
\]
This completes the proof of \cref{thm:main}.

\section{Proof of the localized flexibility and rigidity theorem}\label{sec:flexibility}

\subsection{Strong \texorpdfstring{$L^p$}{Lp} density below the endpoint}

\begin{proof}[Proof of \cref{thm:flexibility-rigidity}\textup{(i)}]
Fix $1\le r<2$.  It suffices to treat a nonzero $v_0\in C_{c,\sigma}^\infty(B;\R^3)$.  Indeed, an arbitrary element of $\mathcal X_r(B)$ can first be approximated by such a field; if the first approximation is zero, add a nonzero compactly supported solenoidal field of arbitrarily small $L^r$ norm.  Choose a ball $U\Subset B$ containing $\supp v_0$.

Define
\[
 A_0=-\nabla v_0-(\nabla v_0)^T+v_0\otimes v_0,
 \qquad
 R_0=\dev A_0,
 \qquad
 p_0=-\frac13\operatorname{tr}A_0.
\]
Then $(v_0,p_0,R_0)$ solves \eqref{eq:NSR} in $U$.  Fix $\delta_0>\|R_0\|_1$ and set
\[
 \delta_q=\delta_0 2^{-4q},
 \qquad
 b_q=2^{-q-3},
 \qquad
 M_q=\sqrt{12}\,\delta_q^{1/2}+b_q.
\]
For $r=1$, choose positive numbers $\ell_q$ such that $\sum_q\ell_q<\eps/2$.  For $1<r<2$, let
\[
 \vartheta=\frac2r-1\in(0,1)
\]
and choose $\ell_q>0$ so small that
\begin{equation}\label{eq:Lp-interpolation-choice}
 \ell_q^\vartheta M_q^{1-\vartheta}<2^{-q-2}\eps.
\end{equation}
Choose also positive $\zeta_q$ with $\sum_q\zeta_q<\infty$, and set
\[
 \kappa_q=2^{-q-3}\|v_0\|_2^2.
\]

Starting from $(u_0,p_0,R_0)=(v_0,p_0,R_0)$, apply \cref{prop:perturbation} inductively with
\[
 \delta=\delta_{q+1},
 \quad \eta_2=b_q,
 \quad \eta_1=\ell_q,
 \quad \eta_{-1}=1,
 \quad \zeta=\zeta_q,
 \quad N=1,
 \quad \varphi_1=v_0,
 \quad \kappa_1=\kappa_q.
\]
Writing $w_{q+1}=u_{q+1}-u_q$, we obtain
\begin{align*}
 \|R_q\|_1&\le\delta_q,\\
 \|w_{q+1}\|_2&\le M_q,\\
 \|w_{q+1}\|_1&<\ell_q,\\
 \left|\int w_{q+1}\cdot v_0\right|&<\kappa_q,\\
 \|p_{q+1}-p_q\|_1&\le4\delta_q+\zeta_q.
\end{align*}
Since $\sum_qM_q<\infty$, the velocities converge strongly in $L^2$; the pressures converge strongly in $L^1$; and the usual limit passage gives a pair $(u,p)$ supported in $U$ and solving \eqref{eq:NS}.

If $r=1$, then $\|u-v_0\|_1<\eps/2$.  If $1<r<2$, interpolation gives
\[
 \|w_{q+1}\|_r
 \le\|w_{q+1}\|_1^\vartheta\|w_{q+1}\|_2^{1-\vartheta}
 <2^{-q-2}\eps
\]
by \eqref{eq:Lp-interpolation-choice}; hence $\|u-v_0\|_r<\eps/2$.  Finally,
\[
 \left|\int(u-v_0)\cdot v_0\right|
 \le\sum_{q=0}^\infty\kappa_q
 =\frac14\|v_0\|_2^2,
\]
so
\[
 \int u\cdot v_0\ge\frac34\|v_0\|_2^2>0.
\]
Thus $u\ne0$.  Combining this construction with the initial smooth approximation proves the theorem.
\end{proof}

\subsection{Negative Sobolev and weak endpoint density}

\begin{proof}[Proof of the strong $H^{-1}$ assertion in \cref{thm:flexibility-rigidity}\textup{(ii)}]
Repeat the proof of \cref{thm:flexibility-rigidity}\textup{(i)}, but at stage $q$ prescribe
\[
 \|w_{q+1}\|_{H^{-1}}<2^{-q-2}\eps
\]
through \cref{prop:perturbation}.  The $L^2$ increment bounds are unchanged, so the iteration still converges strongly in $L^2$, while the sum of the $H^{-1}$ increments is less than $\eps/2$.  The same pairing constraint preserves nontriviality, and an initial $H^{-1}$ approximation by a nonzero smooth compactly supported solenoidal field completes the proof.
\end{proof}

\begin{proof}[Proof of the weak $L^2$ assertion in \cref{thm:flexibility-rigidity}\textup{(ii)}]
Fix $v_0\in C_{c,\sigma}^\infty(B;\R^3)$.  Apply the proof of \cref{thm:flexibility-rigidity}\textup{(i)} with $r=1$ and with the total $L^1$ error tending to zero, while keeping the sequences $\delta_q$ and $b_q$ fixed.  This gives $u_n\in\mathscr S(B)$ such that
\[
 \|u_n-v_0\|_1\to0,
 \qquad
 \sup_n\|u_n\|_2<\infty.
\]
For any $\phi\in L^2(B)$, let $\phi_m$ be a bounded truncation converging to $\phi$ in $L^2$.  Then
\[
 \left|\int_B(u_n-v_0)\cdot\phi\right|
 \le\|u_n-v_0\|_1\|\phi_m\|_\infty
 +\|u_n-v_0\|_2\|\phi-\phi_m\|_2.
\]
First choose $m$ large and then $n$ large.  Thus $u_n\rightharpoonup v_0$ in $L^2$.  Hence every smooth compactly supported solenoidal field lies in the weak closure of $\mathscr S(B)$.  Since these smooth fields are strongly, and therefore weakly, dense in $\mathcal X_2(B)$, the result follows.
\end{proof}

The preceding flexibility range is sharp at the exponent $2$ in a natural and rigid sense.

\begin{lemma}[Compactly supported solenoidal extension of an affine field]\label{lem:affine-solenoidal-extension}
Let $K\Subset\R^3$ be compact and let $A\in\R^{3\times3}$ satisfy $\operatorname{tr}A=0$.  There exists $\Phi_A\in C_{c,\sigma}^\infty(\R^3;\R^3)$ such that
\[
 \Phi_A(x)=Ax
\]
on a neighborhood of $K$.
\end{lemma}

\begin{proof}
Put $V(x)=Ax$.  Since $\diver V=\operatorname{tr}A=0$ and $V$ is homogeneous of degree one, the antisymmetric tensor
\[
 \Omega_{ij}(x)=\frac13\bigl(x_jV_i(x)-x_iV_j(x)\bigr)
\]
satisfies $\partial_j\Omega_{ij}=V_i$.  Indeed,
\[
 \partial_j(x_jV_i)=3V_i+x\cdot\nabla V_i=4V_i,
 \qquad
 \partial_j(x_iV_j)=V_i+x_i\diver V=V_i.
\]
Choose $\chi\in C_c^\infty(\R^3)$ equal to one near $K$ and define
\[
 (\Phi_A)_i=\partial_j(\chi\Omega_{ij}).
\]
Because $\chi\Omega$ is antisymmetric, $\diver\Phi_A=0$.  Where $\chi=1$, one has $\Phi_A=\diver\Omega=Ax$.
\end{proof}

\subsection{The endpoint quadratic obstruction}

\begin{proof}[Proof of \cref{thm:flexibility-rigidity}\textup{(iii)}]
Let $K=\esssupp u$ and let $A\in\Sym_3$ be trace-free.  By \cref{lem:affine-solenoidal-extension}, choose $\Phi_A\in C_{c,\sigma}^\infty$ equal to $Ax$ near $K$.  Testing \eqref{eq:weak-form} with $\Phi_A$ gives
\[
 0=-\left(\int_{\R^3}u\otimes u\,\dd x\right):A.
\]
Thus the symmetric matrix $M=\int u\otimes u$ is orthogonal to every trace-free symmetric matrix, so $M=cI$.  Taking the trace gives $3c=\|u\|_2^2$, proving \eqref{eq:isotropic-moment}.

If $u_n\to v$ strongly in $L^2$, then $u_n\otimes u_n\to v\otimes v$ strongly in $L^1$, so the moment constraint passes to the limit.  To see that the constraint is nontrivial, choose a nonzero $\phi\in C_c^\infty(B)$ and set
\[
 v=(\partial_2\phi,-\partial_1\phi,0).
\]
Then $v\in C_{c,\sigma}^\infty(B)$, but the $(3,3)$ entry of $\int v\otimes v$ is zero while its trace is positive.  Hence its moment is not isotropic.
\end{proof}

\begin{remark}[Sharp flexibility--rigidity transition]
The preceding results give a clean endpoint picture: strong density holds for every exponent below $2$, weak density holds at exponent $2$, and strong density fails there because of the exact quadratic invariant \eqref{eq:isotropic-moment}.
\end{remark}

\section{Affine identities, exact energy, and multiplicity}\label{sec:affine-energy}

\begin{lemma}[Zero mean]\label{lem:zero-mean}
If $u\in L^1(\R^3;\R^3)$ is compactly supported and $\diver u=0$ in distributions, then
\[
 \int_{\R^3}u\,\dd x=0.
\]
\end{lemma}

\begin{proof}
For each $i$,
\[
 \diver(x_i u)=u_i+x_i\diver u=u_i
\]
in distributions.  Since $x_i u$ is compactly supported and integrable, the integral of its divergence is zero.
\end{proof}

\subsection{The full affine stress identity}

\begin{proof}[Proof of \cref{thm:full-affine-identity}]
Let $A\in\R^{3\times3}$ be arbitrary.  Choose $\Phi\in C_c^\infty(\R^3;\R^3)$ equal to $Ax$ on a neighborhood of $\supp u\cup\supp p$.  Testing the full distributional equation gives
\[
 0=-\left(\int_{\R^3}(u\otimes u+pI)\,\dd x\right):A.
\]
Since $A$ is arbitrary, \eqref{eq:full-affine-identity} follows.  Taking the trace gives the pressure identity, and substituting it back gives the isotropic moment identity.  Finally, \eqref{eq:pressure-lower-bound} follows from $|\int p|\le\|p\|_1$.
\end{proof}

The stationary Navier--Stokes scaling is
\begin{equation}\label{eq:stationary-scaling}
 u_{\lambda,x_0}(x)=\lambda u(\lambda(x-x_0)),
 \qquad
 p_{\lambda,x_0}(x)=\lambda^2p(\lambda(x-x_0)).
\end{equation}
It satisfies
\begin{equation}\label{eq:scaling-norms}
 \|u_{\lambda,x_0}\|_2=\lambda^{-1/2}\|u\|_2,
 \qquad
 \|p_{\lambda,x_0}\|_1=\lambda^{-1}\|p\|_1,
 \qquad
 \operatorname{diam}\supp u_{\lambda,x_0}=\lambda^{-1}\operatorname{diam}\supp u.
\end{equation}

\subsection{Exact energy and multiplicity}

\begin{proof}[Proof of \cref{thm:exact-energy}]
Choose a nonempty ball $B\Subset\Omega$.  Fix one nonzero compactly supported solution $(u_*,p_*)$ given by \cref{thm:main}, translated so that
\[
 \supp u_*\cup\supp p_*\subset B(0,R_{\mathrm{supp}}),
 \qquad
 e_*:=\|u_*\|_2>0.
\]
Write $B=B(x_B,R)$.  For an integer $m\ge1$, set $N=m^3$ and
\[
 \lambda_N=\frac{Ne_*^2}{E^2}.
\]
For $m$ sufficiently large, $\lambda_N$ is large and the support radius $R_{\mathrm{supp}}/\lambda_N$ is $O(m^{-3})$.  Choose $N=m^3$ grid points $x_1,\dots,x_N$ in a fixed cube compactly contained in $B$, with mutual spacing comparable to $m^{-1}$.  For large $m$, the supports of
\[
 u_j=u_{\lambda_N,x_j},
 \qquad
 p_j=p_{\lambda_N,x_j}
\]
are pairwise disjoint and contained in $B$.  Set
\[
 u=\sum_{j=1}^N u_j,
 \qquad
 p=\sum_{j=1}^N p_j.
\]
Because the supports are disjoint,
\[
 u\otimes u=\sum_{j=1}^N u_j\otimes u_j,
\]
so the sum is again a stationary solution.  Moreover,
\[
 \|u\|_2^2
 =N\lambda_N^{-1}e_*^2=E^2,
\]
and
\[
 \|p\|_1
 =N\lambda_N^{-1}\|p_*\|_1
 =\frac{\|p_*\|_1}{e_*^2}E^2.
\]
This proves the upper bound in \eqref{eq:pressure-energy-bounds}, with $C_\Omega=\|p_*\|_1/e_*^2$, while the lower bound follows from \eqref{eq:pressure-lower-bound}.

To obtain uncountably many solutions, keep all but one center fixed and move the remaining center in a small open set that preserves disjointness.  Distinct sufficiently small translations of a nonzero compactly supported function are distinct; otherwise the function would be invariant under a nonzero translation and could not have compact support.
\end{proof}

\section{Periodization and evolutionary consequences}\label{sec:periodic-dynamic}

Write $\T^3=\R^3/\Z^3$ and let $Q=(-\tfrac12,\tfrac12)^3$.

\begin{proposition}[Periodization of a compactly supported solution]\label{prop:periodization}
Suppose $(u,p)$ is a stationary solution on $\R^3$ with
\[
 u\in L^2(\R^3;\R^3),
 \qquad
 p\in L^1(\R^3),
\]
and suppose that both supports are compactly contained in $Q$.  Define
\[
 u_{\rm per}(x)=\sum_{n\in\Z^3}u(x-n),
 \qquad
 p_{\rm per}(x)=\sum_{n\in\Z^3}p(x-n).
\]
Then $(u_{\rm per},p_{\rm per})$ descends to a stationary solution on $\T^3$.  Moreover,
\begin{equation}\label{eq:periodization-properties}
 \int_{\T^3}u_{\rm per}\,\dd x=0,
 \qquad
 \|u_{\rm per}\|_{L^2(\T^3)}=\|u\|_{L^2(\R^3)},
 \qquad
 \|p_{\rm per}\|_{L^1(\T^3)}=\|p\|_{L^1(\R^3)}.
\end{equation}
\end{proposition}

\begin{proof}
The sums are locally finite.  Since the supports are compactly contained in one fundamental cube, distinct translates of the supports are disjoint.  Therefore
\[
 u_{\rm per}\otimes u_{\rm per}
 =\sum_{n\in\Z^3}(u\otimes u)(x-n).
\]
Each translated pair solves the stationary equations, and summing the translated distributional identities gives the periodic equation.  The norm identities follow by integration over $Q$.  Finally,
\[
 \int_{\T^3}u_{\rm per}\,\dd x
 =\int_{\R^3}u\,\dd x=0
\]
by \cref{lem:zero-mean}.
\end{proof}

\begin{proof}[Proof of \cref{cor:periodic-endpoint}]
Choose the corresponding whole-space solutions inside a ball compactly contained in $Q$, using \cref{thm:main,thm:flexibility-rigidity,thm:exact-energy}, and apply \cref{prop:periodization}.
\end{proof}

\subsection{Stationary-versus-Leray nonuniqueness}

\begin{proof}[Proof of \cref{thm:stationary-leray}]
Let $u_{\rm in}\ne0$ be any compactly supported stationary solution supplied by the preceding results.  The first solution is
\[
 U^{\rm st}(t,x)=u_{\rm in}(x).
\]
It is a global distributional solution and belongs to $C([0,\infty);L^2)$.  Classical Leray theory provides a global Leray--Hopf solution $U^{\rm L}$ with the same initial datum \cite{Leray1934}; see also \cite{BerselliSpirito2021,Temam2001} for modern treatments, with the latter covering both the whole space and the flat torus.  We use the standard representative described in \cite[Definition~1 and Remark~2]{BerselliSpirito2021}, namely $U^{\rm L}\in C_{\mathrm w}([0,\infty);L^2)$ with
\begin{equation}\label{eq:leray-energy}
 \|U^{\rm L}(t)\|_2^2
 +2\int_0^t\|\nabla U^{\rm L}(s)\|_2^2\,\dd s
 \le\|u_{\rm in}\|_2^2
\end{equation}
holds for every $t\ge0$.  If equality held in the first term at some $t>0$, then \eqref{eq:leray-energy} would force $\nabla U^{\rm L}=0$ for almost every $s\in(0,t)$.  On $\R^3$, an $L^2$ spatially constant field is zero, so $U^{\rm L}(s)=0$ for almost every $s\in(0,t)$.  Weak $L^2$ continuity would then give $U^{\rm L}(0)=0$, contradicting $u_{\rm in}\ne0$.  This proves \eqref{eq:strict-energy-drop}, and in particular the two solutions are distinct at every positive time.

On $\T^3$, periodize the stationary solution and take a periodic Leray--Hopf solution with the same initial datum.  The periodized field has zero mean by \cref{prop:periodization}, and the spatial mean of a periodic weak solution is conserved.  Hence $\nabla U^{\rm L}=0$ on a time interval would force $U^{\rm L}=0$ there; weak continuity again contradicts the nonzero initial datum.  The norm assertions follow from \cref{thm:main,thm:exact-energy,cor:periodic-endpoint}.
\end{proof}

\begin{proof}[Proof of the density assertion in \cref{thm:stationary-leray}]
Every nonzero stationary solution furnished by \cref{thm:flexibility-rigidity} is an initial datum to which \cref{thm:stationary-leray} applies.
\end{proof}

\begin{remark}[Solution class and comparison]
The preceding corollary is not a nonuniqueness result within the Leray--Hopf class: the stationary branch lies outside the energy class.  It is a nonuniqueness statement in the broader class $L_t^\infty L_x^2\cap C_{t,\mathrm w}L_x^2$, obtained from a stationary-versus-dissipative mechanism. 
\end{remark}

\appendix

\section{A self-contained flat symmetric-divergence construction}\label{app:local-kernel}

This appendix proves the precise Euclidean ball result used in \cref{lem:local-sym-div}.  The kernel formula agrees with the flat straight-line formula in \cite[Appendix~A.4.3]{IMOT2025}, which we cite for provenance; none of the mapping or support theorems from that reference is used below.

Fix a ball $U\subset\R^3$ and choose once and for all
\[
 \eta\in C_c^\infty(U),\qquad \int_{\R^3}\eta=1.
\]
All constants below may depend on this fixed choice and on $U$; we denote them by $C_U$.
For $z\ne0$, $y\in\R^3$, write
\[
 r=|z|,\qquad \omega=\frac z{|z|},
\]
and define
\begin{equation}\label{eq:appendix-aeta}
 \mathfrak a_\eta(z,y)
 =\int_r^\infty \eta(y+s\omega)s^2\,\dd s.
\end{equation}
Set
\begin{equation}\label{eq:appendix-Keta}
 \begin{split}
 (K_\eta)_{ij}{}^k(y+z,y)
 ={}&\frac12\mathfrak a_\eta(z,y)
 \frac{z_i\delta_j^k+z_j\delta_i^k}{r^3}\\
 &+\frac12\partial_{z_m}\!\left[
 \mathfrak a_\eta(z,y)
 \frac{z_m(z_i\delta_j^k+z_j\delta_i^k)}{r^3}
 \right]\\
 &-\partial_{z_k}\!\left[
 \mathfrak a_\eta(z,y)\frac{z_i z_j}{r^3}
 \right].
 \end{split}
\end{equation}
For $f\in C_c^\infty(U;\R^3)$ define
\begin{equation}\label{eq:appendix-Seta}
 (S_\eta f)_{ij}(x)
 =\int_{\R^3}(K_\eta)_{ij}{}^k(x,y)f_k(y)\,\dd y.
\end{equation}
The repeated indices in this appendix are summed over $1,2,3$.

\begin{proposition}[Flat local symmetric anti-divergence]\label{prop:flat-sym-div}
The tensor $S_\eta f$ is smooth, symmetric, and compactly supported, and
\begin{equation}\label{eq:appendix-divSeta}
 \diver S_\eta f=f-B_\eta f,
\end{equation}
where
\begin{equation}\label{eq:appendix-beta}
 (b_\eta)_i{}^k(x,y)
 =2\eta(x)\delta_i^k
 +\frac12(x-y)_\ell\partial_\ell\eta(x)\delta_i^k
 -\frac12(x-y)_i\partial_k\eta(x)
\end{equation}
and
\[
 (B_\eta f)_i(x)=\int_{\R^3}(b_\eta)_i{}^k(x,y)f_k(y)\,\dd y.
\]
Moreover,
\begin{equation}\label{eq:appendix-support}
 \supp S_\eta f
 \subset\bigcup_{y\in\supp f,\ y_1\in\supp\eta}[y,y_1]
 \subset\operatorname{co}(\supp f\cup\supp\eta)\Subset U,
\end{equation}
and, for every $y\in U$ and $x\ne y$,
\begin{equation}\label{eq:appendix-kernel-order}
 |K_\eta(x,y)|
 \le C_U|x-y|^{-2}
 \one_{\{|x-y|\le\operatorname{diam}U\}}.
\end{equation}
Consequently,
\begin{equation}\label{eq:appendix-L1}
 \|S_\eta f\|_{L^1(\R^3)}
 \le C_U\|f\|_{L^1(\R^3)}.
\end{equation}
Finally, if $f$ has zero total force and zero total torque, then $B_\eta f=0$.
\end{proposition}

\begin{proof}
The symmetry $(K_\eta)_{ij}{}^k=(K_\eta)_{ji}{}^k$ is immediate from \eqref{eq:appendix-Keta}.  We prove the remaining assertions in four steps.

\smallskip
\noindent\emph{Step 1: support and the order $-1$ bound.}
If $\mathfrak a_\eta(z,y)\ne0$, then for some $s\ge r$ the point
\[
 y_1=y+s\omega
\]
belongs to $\supp\eta$, and
\[
 x=y+z
 =\left(1-\frac rs\right)y+\frac rs y_1\in[y,y_1].
\]
Taking $z$-derivatives does not enlarge this support.  This proves the first inclusion in \eqref{eq:appendix-support}; the second follows from convexity of $U$.  Because both endpoint sets are compactly contained in $U$, their convex hull is also compactly contained in $U$.

For the kernel estimate, fix $y\in U$, which is the only range entering \eqref{eq:appendix-Seta} because $\supp f\subset U$.  Let $D=\operatorname{diam}U$.  On the kernel support, the point $y_1\in\supp\eta\subset U$ found above satisfies $r\le |y_1-y|\le D$.  From \eqref{eq:appendix-aeta},
\begin{equation}\label{eq:appendix-a-bounds}
 |\mathfrak a_\eta(z,y)|\le C_U,
 \qquad
 |\nabla_z\mathfrak a_\eta(z,y)|\le C_U r^{-1}.
\end{equation}
Indeed, the radial derivative is
\[
 \partial_r\mathfrak a_\eta=-r^2\eta(y+r\omega),
\]
whereas an angular derivative is bounded by
\[
 \int_r^D s^3|\nabla\eta(y+s\omega)|\,\dd s;
\]
the conversion from angular to Cartesian derivatives contributes one factor $r^{-1}$.  Substitution into \eqref{eq:appendix-Keta} gives \eqref{eq:appendix-kernel-order}.

\smallskip
\noindent\emph{Step 2: the divergence away from the diagonal.}
Let
\[
 E=z\cdot\nabla_z,
 \qquad q=E\mathfrak a_\eta.
\]
Since the tensor
\[
 V_{ij}{}^k(z)=\frac{z_i\delta_j^k+z_j\delta_i^k}{r^3}
\]
is homogeneous of degree $-2$, one has
\[
 \partial_{z_m}(z_mV_{ij}{}^k)=V_{ij}{}^k.
\]
Thus \eqref{eq:appendix-Keta} may be rewritten, for $z\ne0$, as
\begin{equation}\label{eq:appendix-K-simplified}
 (K_\eta)_{ij}{}^k
 =\left(\mathfrak a_\eta+\frac12q\right)V_{ij}{}^k
 -\partial_{z_k}\!\left(\mathfrak a_\eta\frac{z_i z_j}{r^3}\right).
\end{equation}
Using
\[
 \partial_{z_j}(z_jr^{-3})=0,
 \qquad
 \partial_{z_j}(z_i z_jr^{-3})=z_i r^{-3},
\]
a direct differentiation yields
\begin{equation}\label{eq:appendix-div-away}
 \partial_{z_j}(K_\eta)_{ij}{}^k
 =-\frac12\partial_{z_k}\!\left(\frac{z_iq}{r^3}\right)
 +\frac12\delta_i^k\frac{2q+Eq}{r^3}.
\end{equation}
Keeping $\omega$ fixed while differentiating the lower endpoint in \eqref{eq:appendix-aeta} gives
\begin{equation}\label{eq:appendix-q}
 q=r\partial_r\mathfrak a_\eta
 =-r^3\eta(y+z).
\end{equation}
Substituting \eqref{eq:appendix-q} into \eqref{eq:appendix-div-away} gives, for $x=y+z\ne y$,
\begin{equation}\label{eq:appendix-pointwise-div}
 \partial_{x_j}(K_\eta)_{ij}{}^k(x,y)
 =-(b_\eta)_i{}^k(x,y),
\end{equation}
with $b_\eta$ as in \eqref{eq:appendix-beta}.

\smallskip
\noindent\emph{Step 3: the diagonal term and the finite-rank correction.}
By \eqref{eq:appendix-kernel-order}, $K_\eta(\cdot,y)$ is locally integrable.  Hence
\[
 T_i{}^k:=\partial_{x_j}(K_\eta)_{ij}{}^k+(b_\eta)_i{}^k
\]
is supported at $x=y$.  To determine its order, let $\varphi\in C_c^\infty(\R^3)$ and integrate by parts on $\{|x-y|>\epsilon\}$.  The contribution of the smooth term $(b_\eta)_i{}^k\varphi$ over $B_\epsilon(y)$ is $O(\epsilon^3)$.  On the boundary sphere one has $|K_\eta|\lesssim\epsilon^{-2}$, while $\varphi(x)-\varphi(y)=O(\epsilon)$ and the surface measure is $O(\epsilon^2)$; hence replacing $\varphi(x)$ by $\varphi(y)$ changes the boundary term by $O(\epsilon)$.  Thus $T_i{}^k$ depends only on $\varphi(y)$ and is a multiple of $\delta_y$, with no derivative of $\delta_y$:
\[
 T_i{}^k=c_i{}^k(y)\delta_y.
\]
Choose a compactly supported cutoff $\chi_y$ equal to one on the union of the $x$-supports of $K_\eta(\cdot,y)$ and $(b_\eta)_i{}^k(\cdot,y)$.  Testing the preceding identity against $\chi_y$, the distributional derivative of $K_\eta$ contributes zero, and therefore
\[
 c_i{}^k(y)=\int_{\R^3}(b_\eta)_i{}^k(x,y)\,\dd x.
\]
Using $\int\eta=1$ and integration by parts,
\[
 \int(x-y)_\ell\partial_\ell\eta\,\dd x=-3,
 \qquad
 \int(x-y)_i\partial_k\eta\,\dd x=-\delta_i^k.
\]
Therefore $c_i{}^k=\delta_i^k$, and
\begin{equation}\label{eq:appendix-kernel-div-distribution}
 \partial_{x_j}(K_\eta)_{ij}{}^k(x,y)
 =\delta_i^k\delta_y(x)-(b_\eta)_i{}^k(x,y)
\end{equation}
in distributions.  Integrating against $f_k(y)$ proves \eqref{eq:appendix-divSeta}.

Write
\[
 F_k=\int_{\R^3}f_k(y)\,\dd y,
 \qquad
 M_{ik}=\int_{\R^3}y_i f_k(y)\,\dd y.
\]
Equation \eqref{eq:appendix-beta} gives
\begin{equation}\label{eq:appendix-Bmoments}
 \begin{split}
 (B_\eta f)_i={}&2\eta F_i
 +\frac12\partial_\ell\eta\,(x_\ell F_i-M_{\ell i})\\
 &-\frac12\partial_k\eta\,(x_iF_k-M_{ik}).
 \end{split}
\end{equation}
If $F=0$, then
\begin{equation}\label{eq:appendix-Btorque}
 (B_\eta f)_i
 =\frac12\partial_k\eta\,(M_{ik}-M_{ki}).
\end{equation}
The antisymmetric moments $M_{ik}-M_{ki}$ are exactly the total torques in \eqref{eq:compatibility}; hence compatible data satisfy $B_\eta f=0$.

\smallskip
\noindent\emph{Step 4: smoothness and the $L^1$ estimate.}
The kernel depends linearly on $\eta$ and obeys the translation identity
\begin{equation}\label{eq:appendix-translation-identity}
 (\partial_{x_\ell}+\partial_{y_\ell})K_\eta(x,y)
 =K_{\partial_\ell\eta}(x,y).
\end{equation}
The identity is immediate for $x\ne y$ from simultaneous translation of $x$ and $y$, which leaves $z=x-y$ fixed and differentiates only the argument of $\eta$ in \eqref{eq:appendix-aeta}.  Since both kernels are locally integrable, the identity also holds in distributions across the diagonal.  Distributional integration by parts in $y$ therefore gives
\begin{equation}\label{eq:appendix-derivative-transfer}
 \partial_{x_\ell}S_\eta f
 =S_\eta(\partial_\ell f)+S_{\partial_\ell\eta}f.
\end{equation}
Iterating \eqref{eq:appendix-derivative-transfer} expresses every $x$-derivative of $S_\eta f$ as a finite sum of integral operators of the form \eqref{eq:appendix-Seta}, with derivatives of $\eta$ and $f$ in place of $\eta$ and $f$.  Each corresponding kernel satisfies the same locally integrable $|x-y|^{-2}$ bound.  More explicitly, the integral over $|x-y|<\rho$ is bounded uniformly in $x$ by $C\rho$, while on $|x-y|\ge\rho$ the kernel is smooth in $x$ and dominated convergence applies.  First letting $x$ vary and then $\rho\downarrow0$ proves continuity of every derivative.  Hence $S_\eta f\in C_c^\infty(U;\Sym_3)$.

Finally, Fubini's theorem and \eqref{eq:appendix-kernel-order} yield
\[
 \begin{split}
 \|S_\eta f\|_1
 &\le\int_U|f(y)|\left(\int_U|K_\eta(x,y)|\,\dd x\right)\dd y\\
 &\le C_U\int_U|f(y)|
 \left(\int_0^D r^{-2}r^2\,\dd r\right)\dd y
 \le C_U\|f\|_1,
 \end{split}
\]
which is \eqref{eq:appendix-L1}.  This estimate uses only local integrability of the order $-1$ kernel, not an endpoint estimate for an order-zero Calder\'on--Zygmund operator.
\end{proof}

\section*{Acknowledgments}

The author acknowledges the use of ChatGPT (OpenAI) as an auxiliary
tool in the preparation of Lemma \ref{lem:dipole}. All mathematical arguments and
conclusions in that lemma were independently checked and verified by
the author, who assumes full responsibility for its content.


\begin{thebibliography}{99}

\bibitem{AlbrittonBrueColombo2022}
D. Albritton, E. Bru\`e, and M. Colombo,
\emph{Non-uniqueness of Leray solutions of the forced Navier--Stokes equations},
Ann. of Math. (2) \textbf{196} (2022), 415--455.

\bibitem{BDIS2015}
T. Buckmaster, C. De Lellis, P. Isett, and L. Sz\'ekelyhidi Jr.,
\emph{Anomalous dissipation for $1/5$-H\"older Euler flows},
Ann. of Math. (2) \textbf{182} (2015), 127--172.

\bibitem{BerselliSpirito2021}
L. C. Berselli and S. Spirito,
\emph{On the existence of Leray--Hopf weak solutions to the Navier--Stokes equations},
Fluids \textbf{6} (2021), Paper No. 42.

\bibitem{BrueColomboKumar2024}
E. Bru\`e, M. Colombo, and A. Kumar,
\emph{Flexibility of two-dimensional Euler flows with integrable vorticity},
arXiv:2408.07934, 2024.

\bibitem{BuckmasterVicol2019}
T. Buckmaster and V. Vicol,
\emph{Nonuniqueness of weak solutions to the Navier--Stokes equation},
Ann. of Math. (2) \textbf{189} (2019), 101--144.

\bibitem{ChengKwonLi2021}
X. Cheng, H. Kwon, and D. Li,
\emph{Non-uniqueness of steady-state weak solutions to the surface quasi-geostrophic equations},
Comm. Math. Phys. \textbf{388} (2021), 1281--1295.

\bibitem{CheskidovHou2026}
A. Cheskidov and H. Hou,
\emph{On non-uniqueness of mild solutions and stationary singular solutions to the Navier--Stokes equations},
arXiv:2603.03666, 2026.

\bibitem{CheskidovLuo2022}
A. Cheskidov and X. Luo,
\emph{Sharp nonuniqueness for the Navier--Stokes equations},
Invent. Math. \textbf{229} (2022), 987--1054.

\bibitem{CheskidovZengZhang2025}
A. Cheskidov, Z. Zeng, and D. Zhang,
\emph{Global dissipative solutions of the 3D Naiver--Stokes and MHD equations},
arXiv:2503.05692, 2025.

\bibitem{ColomboColomboKumar2025}
M. Colombo, R. Colombo, and A. Kumar,
\emph{A convex integration scheme for the continuity equation past the Sobolev embedding threshold},
arXiv:2504.03578, 2025.

\bibitem{DaiGiriRadu2024}
M. Dai, V. Giri, and R.-O. Radu,
\emph{An Onsager-type theorem for SQG},
J. Eur. Math. Soc., forthcoming; arXiv:2407.02582, 2024.

\bibitem{DaneriSzekelyhidi2017}
S. Daneri and L. Sz\'ekelyhidi Jr.,
\emph{Non-uniqueness and h-principle for H\"older-continuous weak solutions of the Euler equations},
Arch. Ration. Mech. Anal. \textbf{224} (2017), 471--514.

\bibitem{DLS2009}
C. De Lellis and L. Sz\'ekelyhidi Jr.,
\emph{The Euler equations as a differential inclusion},
Ann. of Math. (2) \textbf{170} (2009), 1417--1436.

\bibitem{DLS2013}
C. De Lellis and L. Sz\'ekelyhidi Jr.,
\emph{Dissipative continuous Euler flows},
Invent. Math. \textbf{193} (2013), 377--407.

\bibitem{DLS2014}
C. De Lellis and L. Sz\'ekelyhidi Jr.,
\emph{Dissipative Euler flows and Onsager's conjecture},
J. Eur. Math. Soc. \textbf{16} (2014), 1467--1505.

\bibitem{EncisoPenafielPeraltaExtension2024}
A. Enciso, J. Pe\~nafiel-Tom\'as, and D. Peralta-Salas,
\emph{An extension theorem for weak solutions of the 3D incompressible Euler equations and applications to singular flows},
Forum Math. Pi \textbf{13} (2025), e21.

\bibitem{EncisoPenafielPeraltaMHD2025}
A. Enciso, J. Pe\~nafiel-Tom\'as, and D. Peralta-Salas,
\emph{H\"older continuous dissipative solutions of ideal MHD with nonzero helicity},
arXiv:2507.23749, 2025.

\bibitem{EncisoPenafielPeraltaSteady2025}
A. Enciso, J. Pe\~nafiel-Tom\'as, and D. Peralta-Salas,
\emph{Steady 3D Euler flows via a topology-preserving convex integration scheme},
arXiv:2501.13632, 2025.

\bibitem{Fujii2026}
M. Fujii,
\emph{Sharp non-uniqueness for the Navier--Stokes equations in scaling critical spaces},
arXiv:2602.19846, 2026.


\bibitem{FujiiL2}
M.~Fujii,
\emph{Non-unique $L^2$ solutions to the stationary Navier--Stokes equations on the whole plane},
preprint, arXiv:2608.17456, 2026.

\bibitem{Galdi2011}
G. P. Galdi,
\emph{An Introduction to the Mathematical Theory of the Navier--Stokes Equations: Steady-State Problems},
2nd ed., Springer Monographs in Mathematics, Springer, 2011.

\bibitem{GiardiSzekelyhidi2026}
M. Giardi and L. Sz\'ekelyhidi Jr.,
\emph{$C^{1/5^-}$ convex integration solutions of ideal MHD},
arXiv:2604.12091, 2026.

\bibitem{GiriKwonNovack2026}
V. Giri, H. Kwon, and M. Novack,
\emph{Non-conservation of a generalized helicity in the Euler equations},
arXiv:2601.05869, 2026.

\bibitem{GiriRadu2024}
V. Giri and R.-O. Radu,
\emph{The Onsager conjecture in 2D: a Newton--Nash iteration},
Invent. Math. \textbf{238} (2024), 691--768.

\bibitem{GismondiMaPathakRadu2026}
N. Gismondi, K. (Mark) Ma, M. Pathak, and A. F. Radu,
\emph{Non-unique solutions to the periodic gKdV equation},
arXiv:2606.06916, 2026.

\bibitem{Gromov1973}
M. L. Gromov,
\emph{Convex integration of differential relations. I},
Math. USSR-Izv. \textbf{7} (1973), 329--343.

\bibitem{Gromov1986}
M. Gromov,
\emph{Partial Differential Relations},
Ergebnisse der Mathematik und ihrer Grenzgebiete, vol. 9, Springer-Verlag, Berlin, 1986.

\bibitem{IMOT2025}
P. Isett, Y. Mao, S.-J. Oh, and Z. Tao,
\emph{Integral formulas for under/overdetermined differential operators via recovery on curves and the finite-dimensional cokernel condition I: General theory},
arXiv:2509.04617, 2025.

\bibitem{Isett2018}
P. Isett,
\emph{A proof of Onsager's conjecture},
Ann. of Math. (2) \textbf{188} (2018), 871--963.

\bibitem{Leray1934}
J. Leray,
\emph{Sur le mouvement d'un liquide visqueux emplissant l'espace},
Acta Math. \textbf{63} (1934), 193--248.

\bibitem{LiQu2026}
H. Li and P. Qu,
\emph{Non-uniqueness for the hypo-dissipative compressible 3D magnetohydrodynamic equations},
arXiv:2606.21040, 2026.

\bibitem{LooiIsett2024}
S.-Z. Looi and P. Isett,
\emph{A proof of Onsager's conjecture for the SQG equation},
arXiv:2407.02578, 2024.

\bibitem{Luo2019}
X. Luo,
\emph{Stationary solutions and nonuniqueness of weak solutions for the Navier--Stokes equations in high dimensions},
Arch. Ration. Mech. Anal. \textbf{233} (2019), 701--747.

\bibitem{LuoTiti2020}
T. Luo and E. S. Titi,
\emph{Non-uniqueness of weak solutions to hyperviscous Navier--Stokes equations: on sharpness of J.-L. Lions exponent},
Calc. Var. Partial Differential Equations \textbf{59} (2020), Paper No. 92.

\bibitem{MaoQuElastodynamics2025}
S. Mao and P. Qu,
\emph{The null condition in elastodynamics leads to non-uniqueness},
arXiv:2502.07521, 2025.

\bibitem{MaoQuLame2025}
S. Mao and P. Qu,
\emph{Non-uniqueness for the nonlinear dynamical Lam\'e system},
J. Differential Equations \textbf{446} (2025), Paper No. 113603.

\bibitem{MiaoNieYeFiniteEnergy2024}
C. Miao, Y. Nie, and W. Ye,
\emph{Non-uniqueness of weak solutions to the Navier--Stokes equations in $\R^3$},
arXiv:2412.10404, 2024.

\bibitem{MiaoNieYe2024}
C. Miao, Y. Nie, and W. Ye,
\emph{Sharp non-uniqueness for the Navier--Stokes equations in $\R^3$},
arXiv:2412.09637, 2024.

\bibitem{MiaoNieYe2025}
C. Miao, Y. Nie, and W. Ye,
\emph{On Onsager-type conjecture for the Els\"asser energies of the ideal MHD equations},
arXiv:2504.06071, 2025.

\bibitem{Nash1954}
J. Nash,
\emph{$C^1$ isometric imbeddings},
Ann. of Math. (2) \textbf{60} (1954), 383--396.

\bibitem{QuZhang2026}
P. Qu and M. Zhang,
\emph{Non-uniqueness of weak solutions to the 3-D stationary MHD equations in Besov space with negative regularity index},
arXiv:2606.16141, 2026.

\bibitem{Scheffer1993}
V. Scheffer,
\emph{An inviscid flow with compact support in space-time},
J. Geom. Anal. \textbf{3} (1993), 343--401.

\bibitem{Shnirelman1997}
A. Shnirelman,
\emph{On the nonuniqueness of weak solution of the Euler equation},
Comm. Pure Appl. Math. \textbf{50} (1997), 1261--1286.

\bibitem{Temam2001}
R. Temam,
\emph{Navier--Stokes Equations: Theory and Numerical Analysis},
AMS Chelsea Publishing, Providence, RI, 2001.

\bibitem{ZhaoActiveScalar2024}
X. Zhao,
\emph{An Onsager-type theorem for general 2D active scalar equations},
arXiv:2412.11094, 2024.

\end{thebibliography}
\end{document}